\documentclass{article}
\usepackage{mathrsfs}
\usepackage[utf8]{inputenc}
\usepackage[letterpaper,margin=1.0in]{geometry}

\usepackage{amsmath}\allowdisplaybreaks
\usepackage{amsfonts,bm}
\usepackage{amssymb}

\def\eqref#1{equation~(\ref{#1})}

\def\1{\bf{1}}

\usepackage{amsthm}
\theoremstyle{plain}

\makeatletter
\def\Ddots{\mathinner{\mkern1mu\raise\p@
\vbox{\kern7\p@\hbox{.}}\mkern2mu
\raise4\p@\hbox{.}\mkern2mu\raise7\p@\hbox{.}\mkern1mu}}
\makeatother

\makeatletter
\newcommand*{\rom}[1]{\expandafter\@slowromancap\romannumeral #1@}
\makeatother

\usepackage{color}
\usepackage{hyperref}
\usepackage{enumerate}
\usepackage{enumitem}
\usepackage{algorithm}
\usepackage{algorithmic}
\usepackage[numbers,sort,compress]{natbib}
\usepackage{url}
\usepackage{microtype}
\usepackage{algorithm}
\usepackage{algorithmic}
\usepackage{hyperref}
\usepackage{url}
\usepackage{microtype}
\usepackage{enumitem}
\usepackage{graphicx}
\usepackage{xspace}
\usepackage{booktabs} 
\usepackage{makecell}

\theoremstyle{plain}
\newtheorem{theorem}{Theorem}[section]

\newtheorem{lemma}[theorem]{Lemma}

\theoremstyle{definition}
\newtheorem{definition}[theorem]{Definition}
\newtheorem{assumption}[theorem]{Assumption}
\newtheorem{remark}[theorem]{Remark}

\title{Decentralized Non-convex Stochastic Optimization with Heterogeneous Variance}
\author{
Hongxu Chen, Ke Wei, and Luo Luo \\[0.5em]
School of Data Science, Fudan University, Shanghai, China.}

\begin{document}

\maketitle

\begin{abstract}
Decentralized optimization is critical for solving large-scale machine learning problems over distributed networks, where multiple nodes collaborate through local communication. In practice, the variances of stochastic gradient estimators often differ across nodes, yet their impact on algorithm design and complexity remains unclear. To address this issue, we propose D-NSS, a decentralized algorithm with node-specific sampling, and establish its sample complexity depending on the arithmetic mean of local standard deviations, achieving tighter bounds than existing methods that rely on the worst-case or quadratic mean. We further derive a matching sample complexity lower bound under heterogeneous variance, thereby proving the optimality of this dependence. Moreover, we extend the framework with a variance reduction technique and develop D-NSS-VR, which under the mean-squared smoothness assumption attains an improved sample complexity bound while preserving the arithmetic-mean dependence. Finally, numerical experiments validate the theoretical results and demonstrate the effectiveness of the proposed algorithms.
\end{abstract}

\section{Introduction}

With the increasing demand for data processing and computation, distributed optimization has become an essential tool for large-scale machine learning problems. 
In particular, decentralized optimization allows nodes to communicate with their neighbors, improving robustness and avoiding communication bottlenecks.
In this paper, we consider the following decentralized non-convex stochastic optimization problem:
\begin{align}
\label{eq_problem}
\min_{x \in \mathbb{R}^d} f(x) := \frac{1}{m} \sum_{i=1}^m f_i(x),
\end{align}
where $f_i(x) = \mathbb{E}_{\xi_i \sim \mathcal{D}_i} \left[ F(x; \xi_i) \right]$ is a possibly non-convex objective defined by the local data distribution $\mathcal{D}_i$ at node $i$ and $\xi_i$ is the corresponding random index.

Heterogeneity in the local data distributions $\mathcal{D}_i$ leads to differences in the variances of stochastic gradient estimators across nodes.
Specifically, for first-order methods applied to problem (\ref{eq_problem}), node $i$ can only access an unbiased stochastic gradient estimator $g_i(x; \xi_i)$ satisfying 
\begin{align}
\label{eq_g}
\mathbb{E}_{\xi_i} \left[ \|g_i(x; \xi_i) - \nabla f_i(x)\|^2 \right] \leq \sigma_i^2,
\end{align}
where $\sigma_i > 0$ may vary significantly across nodes due to the non-IID data distribution.
However, most existing decentralized optimization methods \cite{lian2017can,tang2018d,assran2019stochastic,yuan2022revisiting,lu2023decentralized} typically assume uniform variance across nodes, resulting in complexity bounds that depend on the worst-case standard deviation $\sigma_{\max}$.
Xin et al. \cite{xin2021improved,xin2021hybrid,xin2021stochastic} analyze node-wise noise and derive bounds based on the quadratic mean of standard deviations, $\bar{\sigma}_{\mathrm{QM}} := \sqrt{(\sigma_1^2 + \dots + \sigma_m^2)/m}$, but their algorithms still assume identical sample sizes across nodes.
Consequently, how variance heterogeneity affects the convergence and complexity of decentralized algorithms remains unclear, raising two fundamental questions:
\begin{itemize}
    \item[(i)] \textit{How to design efficient decentralized algorithms under heterogeneous variance?}
    \item[(ii)] \textit{What is the optimal sample complexity in this setting?}
\end{itemize}

To answer these questions, we investigate the problem from a sampling perspective and propose a more efficient decentralized algorithm with node-specific sampling, called D-NSS (Decentralized Optimization with Node-Specific Sampling), which achieves the sample complexity (i.e., the total number of stochastic gradient evaluations) of
\begin{align*}
O \left( \frac{\Delta L\, \bar{\sigma}_{\mathrm{AM}}^2}{\epsilon^4} + \frac{m \Delta L}{\epsilon^2} \right)
\end{align*}
for finding the $\epsilon$-stationary point, 
where $\bar{\sigma}_{\mathrm{AM}} = \frac{1}{m}\sum_{i=1}^m\sigma_i$ denotes the arithmetic mean of standard deviations.
This dependence is tighter than the worst-case or quadratic-mean bounds in previous works. In highly heterogeneous regimes such  $\bar{\sigma}_{\mathrm{QM}}= \Theta(\sqrt{m} \bar{\sigma}_{\mathrm{AM}})$), it yields an $O(m)$ improvement in sample complexity.
In addition, we establish a lower bound for decentralized optimization with heterogeneous variance, thereby proving the optimality of D-NSS.
Furthermore, under the mean-squared smoothness assumption, we incorporate variance reduction into the proposed scheme and develop D-NSS-VR, which achieves a sample complexity of
\begin{align*}
O \left( \frac{\Delta \bar{L} \bar{\sigma}_{\mathrm{AM}}}{\epsilon^3} + \frac{\bar{\sigma}_{\mathrm{AM}}^2}{\epsilon^2} +  \frac{ \sqrt{m}\Delta \bar{L}}{\epsilon^2} + m \right),
\end{align*}
where $\Delta$ denotes the initial optimality gap, $L$ and $\bar{L}$ are smoothness constants.

\subsection{Related Work}
Decentralized optimization has developed rapidly over the past decade. 
A classical method is decentralized gradient descent (DGD), which performs gradient descent on local objectives followed by one round of communication.
Under data heterogeneity, DGD requires diminishing step sizes to guarantee convergence, which results in slow convergence rates \cite{nedic2009distributed,yuan2016convergence}.
Gradient tracking techniques \cite{di2016next,pu2021distributed,qu2017harnessing,scutari2019distributed} improve convergence by enabling each node to track the global gradient through additional communication of gradient information.
\citet{shi2015extra} propose EXTRA, a first-order method that introduces a correction term to achieve exact convergence with constant step sizes.
\citet{nedic2017achieving} combine gradient tracking with inexact gradient methods and establish the linear convergence of the DIGing algorithm under strong convexity and smoothness over time-varying graphs. 
To further improve communication efficiency, Chebyshev-accelerated multi-step communication methods have been developed to achieve optimal communication complexity in decentralized optimization \cite{scaman2017optimal, kovalev2020optimal, ye2023multi}.

For decentralized non-convex stochastic optimization, \citet{lian2017can} show that decentralized stochastic gradient descent can find an $\epsilon$-stationary point (i.e., $\mathbb{E}[\|\nabla f(x)\|^2] \le \epsilon^2$) with a sample complexity of $O(\epsilon^{-4})$. The same complexity guarantees are obtained for primal-dual algorithms by \citet{yi2022primal}. Further improvements are obtained by incorporating multi-step communication techniques, as shown in the works of \citet{lu2021optimal, lu2023decentralized} and \citet{yuan2022revisiting}, resulting in nearly optimal sample and communication complexities.
When the stochastic gradients additionally satisfy the mean-squared smoothness condition, variance reduction methods yield improved sample complexity. 
\citet{sun2020improving} propose D-GET by integrating variance reduction with gradient tracking, while \citet{pan2020d} extend the SPIDER-SFO method \cite{fang2018spider} to the decentralized setting. \citet{xin2021hybrid} develop GT-HSGD, a single-loop method based on STORM \cite{cutkosky2019momentum}, and establish a network-independent $O(\epsilon^{-3})$ sample complexity.

However, most of these results rely on the assumption of uniform variance across nodes, and their complexity bounds depend on the worst-case variance $\sigma_{\max}$. Xin et al. \cite{xin2021improved,xin2021hybrid,xin2021stochastic} consider heterogeneous variance in decentralized stochastic optimization and derive an upper bound that depends on the quadratic mean of standard deviations $\bar{\sigma}_{\mathrm{QM}}$, while \citet{deng2025decentralized} obtain similar results in the context of manifold optimization. It remains an open problem whether more efficient algorithms can be designed under heterogeneous variance.

Moreover, theoretical lower bounds are essential for characterizing the performance limits of algorithms. 
For distributed optimization, \citet{arjevani2015communication} derive iteration lower bounds under deterministic convex setting. \citet{scaman2017optimal} analyze the lower complexity bounds for both centralized and decentralized algorithms in the smooth and strongly convex setting. 
In the non-convex case, \citet{arjevani2023lower} establish the complexity lower bound for single-machine non-convex stochastic optimization. \citet{lu2021optimal, lu2023decentralized} extend this result to the decentralized setting, though their analysis relies on a specific linear communication topology. \citet{yuan2022revisiting} construct more general network structures and prove lower bounds under broader conditions, including the special case where the objective satisfies the Polyak–Łojasiewicz (PL) condition. \citet{huang2022optimal} extend the analysis to the time-varying networks. 
In addition, \citet{huang2022lower} and \citet{he2023lower} investigate lower bounds under communication compression.

\section{D-NSS Algorithm}
In this section, we design D-NSS (Decentralized Optimization with Node-Specific Sampling), a sample-efficient algorithm under variance heterogeneity, and analyze its sample and communication complexities.

\subsection{Algorithm Design}
The design of the algorithm is motivated by formulating each iteration as a sample minimization problem with a target accuracy. If each node draws $B_i$ stochastic gradients per iteration, the mean squared error of the global averaged gradient estimator is given by
\begin{align}
\label{eq_2.1}
&\quad \mathbb{E} \left[ \left\| \frac{1}{m} \sum_{i=1}^m \frac{1}{B_i} \sum_{j=1}^{B_i} g_i(x^{(t)}; \xi_{ij}) - \nabla f(x^{(t)}) \right\|^2 \right] \notag\\
&= \frac{1}{m^2} \sum_{i=1}^m \frac{1}{B_i^2} \sum_{j=1}^{B_i} \mathbb{E} \left[ \left\| g_i(x^{(t)}; \xi_{ij}) - \nabla f_i(x^{(t)}) \right\|^2 \right] \notag \\
& \leq \frac{1}{m^2} \sum_{i=1}^m \frac{\sigma_i^2}{B_i},
\end{align}
where the last step holds by \eqref{eq_g}.
We desire to achieve the estimation accuracy of $\epsilon^2$ based on \eqref{eq_2.1}.
A simple choice for batch size setting is $B_i = \sigma_i^2 / (m\epsilon^2)$, yielding a total sample size of $\bar{\sigma}_{\mathrm{QM}}^2 / \epsilon^2$. Although \citet{xin2021stochastic} use a different choice with $B_i = \bar{\sigma}_{\mathrm{QM}}^2 / (m\epsilon^2)$, the total sample size per iteration remains the same. A natural question is how to minimize the total number of samples under the given accuracy. We formulate it as the following optimization problem:
\begin{align}
\label{eq_2.2}
\min_{B_i} \quad & \sum_{i=1}^m B_i \quad \text{s.t.} \quad \frac{1}{m^2} \sum_{i=1}^m \frac{\sigma_i^2}{B_i} \leq \epsilon^2, \quad B_i > 0.
\end{align}
It is a convex problem and its optimal solution can be characterized by the KKT conditions. To handle the inequality constraint, we introduce the Lagrange multiplier $\lambda \ge 0$ and define the Lagrangian as
\begin{align*}
\mathcal{L}(B,\lambda) = \sum_{i=1}^{m} B_i + \lambda \left( \sum_{i=1}^{m} \frac{\sigma_i^{2}}{B_i} - m^{2}\epsilon^{2} \right).
\end{align*}
By the dual feasibility and first-order optimality, we have
\begin{align}
\label{eq_Bstar}
B_i^{\star} = \frac{\sigma_i \sum_{j=1}^{m} \sigma_j}{m^2 \epsilon^2}, \qquad \text{where}\quad i = 1, \dots, m.
\end{align}
Therefore, the minimal total number of samples to achieve an $\epsilon$-accurate estimation is $\sum_{i=1}^m B_i^{\star} = \bar{\sigma}_{\mathrm{AM}}^2 / \epsilon^2$, which depends on the arithmetic mean of the standard deviations.

As discussed above, existing decentralized optimization algorithms \cite{lian2017can,yi2022primal,lu2021optimal, lu2023decentralized, yuan2022revisiting, xin2021stochastic} typically adopt uniform sampling across nodes, which ignores substantial variance heterogeneity and leads to suboptimal dependence in sample complexity.

\begin{algorithm}[t]
\caption{D-NSS}
\label{algo:d-nss}
\textbf{Input}: Initial point $x^0 \in \mathbb{R}^d$, step size $\eta > 0$, communication matrix $W$, communication rounds $R_t$ at iteration $t$, batch size $B_i$ at node $i$ \\
\textbf{Initialize}: $s_i^{-1} = y_i^{-1} = 0$ for all $i$
\begin{algorithmic}[1]
\FOR{$t = 0, 1, \dots, T - 1$}
  \FOR{$i = 1, \dots, m$ \textbf{in parallel}}
    \STATE Sample i.i.d. mini-batch $\mathcal{S}_i^t = \{\xi_{i,1}, \dots, \xi_{i,B_i}\}$
    \STATE $y_i^t = \frac{1}{B_i} \sum_{j=1}^{B_i} g_i(x_i^t; \xi_{i,j})$
    \STATE $s_i^t = \texttt{FastMix}(\{s_i^{t-1} + y_i^t - y_i^{t-1}\}_{i=1}^m, W, R_t)$
    \STATE $x_i^{t+1} = \texttt{FastMix}(\{x_i^t - \eta s_i^t\}_{i=1}^m, W, R_t)$
  \ENDFOR
\ENDFOR
\end{algorithmic}
\textbf{Output}: $x_{i,\text{out}}$ uniformly sampled from $\{x_i^0, x_i^1, \dots, x_i^T\}$
\end{algorithm}

\begin{algorithm}[t]
\caption{\texttt{FastMix}$(\{\phi_i\}_{i=1}^m, W, R)$}
\label{algo:fastmix}
\textbf{Input}: Initial value $\{\phi_i\}_{i=1}^m$, communication matrix $W$, number of rounds $R$ \\
\textbf{Initialize}: $\eta = \frac{1 - \sqrt{1 - \lambda_2^2(W)}}{1 + \sqrt{1 - \lambda_2^2(W)}}$,\nolinebreak\ $z_i^0 = z_i^{-1} = \phi_i$.
\begin{algorithmic}[1]
\FOR{$r = 0, 1, \dots, R - 1$}
  \STATE $z_i^{r+1} = (1 + \eta) \sum_{j=1}^m W_{ij} z_j^r - \eta z_i^{r-1}$
\ENDFOR
\end{algorithmic}
\textbf{Output}: $z_i^R$
\end{algorithm}

Based on the optimal sampling strategy given in \eqref{eq_Bstar}, we propose D-NSS, a decentralized algorithm that allocates node-specific batch sizes, described in Algorithm~\ref{algo:d-nss}. At iteration $t$, node $i$ holds a local variable $x_i^t \in \mathbb{R}^d$ and computes a stochastic gradient estimator $y_i^t$ using a mini-batch of size $B_i$, where $B_i$ is determined by the local gradient noise standard deviation $\sigma_i$.  The local variable $x_i^t$ is then updated using a gradient tracking variable $s_i^t$, which combines $y_i^t$ with information communicated from the neighboring nodes $\mathcal{N}_i$. Fast consensus is achieved through the multi-consensus step \texttt{FastMix} (Algorithm~\ref{algo:fastmix}). 
With this design, D-NSS attains the theoretically optimal sample complexity and is suited for decentralized settings with heterogeneous variance.

\subsection{Complexity Analysis of D-NSS}
To analyze the complexity of the D-NSS algorithm, we first present several standard assumptions in decentralized stochastic optimization.
\begin{assumption}
\label{ass_2.1}
The objective function $f$ is lower bounded, i.e.,
$\inf_{x\in{\mathbb R}^d} f(x) > -\infty$.
\end{assumption}

In addition, we denote the initial optimal function value gap by $\Delta = f(x^0) - \inf_{x\in{\mathbb R}^d} f(x)$, where $x^0$ is the initial point of the algorithm.
It holds $\Delta < +\infty$ under Assumption~\ref{ass_2.1}.

\begin{assumption}
\label{ass_2.2}
For each node $i$, the stochastic gradient $g_i(x, \xi)$ satisfies unbiasedness and bounded variance as
\begin{align*}
& \mathbb{E}_{\xi} [g_i(x; \xi)] = \nabla f_i(x) \quad\text{and} \quad \mathbb{E}_{\xi} \left[ ||g_i(x; \xi) - \nabla f_i(x)||^2 \right] \leq \sigma_i^2,
\end{align*}
where $\sigma_i > 0$.
\end{assumption}
Assumption~\ref{ass_2.2} captures the variance heterogeneity across nodes, which is central to the problem studied in this work.

\begin{assumption}
\label{ass_2.3}
The global objective $f$ is $L$-smooth, and each local function $f_i$ is $mL$-smooth, i.e., for all $x, y \in \mathbb{R}^d$, we have
\begin{align*}
\|\nabla f(x) - \nabla f(y)\|^2 \leq L^2 \|x - y\|^2,
\end{align*}
and for any node $i$ and $x, y \in \mathbb{R}^d$, we have
\begin{align*}
\|\nabla f_i(x) - \nabla f_i(y)\|^2 \leq m^2 L^2 \|x - y\|^2.
\end{align*}
\end{assumption}
\begin{remark} 
We emphasize our Assumption~\ref{ass_2.3} is weaker than the smoothness condition in existing works that requires every~$f_i$ satisfying
$\|\nabla f_i(x) - \nabla f_i(y)\|^2 \leq L^2\|x-y\|^2$ \cite{lu2021optimal, lu2023decentralized, yuan2022revisiting, xin2021stochastic}.
See Appendix~\ref{app_a} for details.
\end{remark}

Finally, we introduce the assumption on the communication matrix $W$, which is standard in the decentralized optimization \cite{ye2023multi,bai2024complexity}.
\begin{assumption}
\label{ass_2.4}
Let $W \in \mathbb{R}^{m \times m}$ be the communication matrix, satisfying:
\begin{itemize}
\item[(a)] $W$ is symmetric and element-wise nonnegative, with $W_{ij} \ne 0$ if and only if nodes $i$ and $j$ are connected or $i = j$;
\item[(b)] $0 \preceq W \preceq I$ and $W^\top \mathbf{1}_m = W \mathbf{1}_m = \mathbf{1}_m$; moreover, the null space of $(I - W)$ is $\operatorname{span}(\mathbf{1})$.
\end{itemize}
\end{assumption}
Assumption~\ref{ass_2.4} indicates that $1-\lambda_2(W) > 0$, where $\lambda_2(W)$ is the second largest eigenvalue of $W$.
We define $\chi := 1-\lambda_2(W)$.

Following theorem provides upper bounds on the sample and communication complexity required by Algorithm~\ref{algo:d-nss} to reach an $\epsilon$-stationary point.

\begin{theorem}
\label{thm_2.1}
Under Assumptions~\ref{ass_2.1}--\ref{ass_2.4}, consider Algorithm~\ref{algo:d-nss} with the following parameter choices:
\begin{align*}
& \eta = \frac{1}{2L}, \quad
B_i = \left\lceil \frac{16\sigma_i \sum_{j=1}^{m} \sigma_j}{m^2 \epsilon^2} \right\rceil, \quad 
T = \left\lceil \frac{32 \Delta L}{\epsilon^2} \right\rceil, \\
& R_0 = O \left( \frac{1}{\sqrt{\chi}} \log\left(\frac{m}{\epsilon}\right) \right), \;
\text{and} \; R_t = O \left( \frac{1}{\sqrt{\chi}} \log(m) \right),
\end{align*}
then the output of the algorithm is an $\epsilon$-stationary point satisfying
$\mathbb{E} \left[ \|\nabla f(x_{i,\mathrm{out}})\|^2 \right] \leq \epsilon^2$.
The sample complexity is upper bounded by
\begin{align*}
O \left( \frac{\Delta L \bar{\sigma}_{\mathrm{AM}}^2}{\epsilon^4} + \frac{m \Delta L}{\epsilon^2} \right),
\end{align*}
and the communication complexity is bounded by
\begin{align*}
\tilde{O} \left( \frac{\Delta L}{\sqrt{\chi} \epsilon^2} \right).
\end{align*}
\end{theorem}

In Table~\ref{table_1}, we compare the result with representative existing algorithms. Since all these methods achieve near-optimal (up to a logarithmic factor) communication complexity, we focus on their sample complexity. D-NSS achieves a dependence on the arithmetic mean of standard deviations $\bar{\sigma}_{\mathrm{AM}}$, which is tighter than the worst-case or quadratic-mean bounds in previous methods. Under significant variance heterogeneity, the parameters can satisfy the relation $\sigma_{\max} = \Theta(\sqrt{m} \, \bar{\sigma}_{\mathrm{QM}}) = \Theta(m \, \bar{\sigma}_{\mathrm{AM}})$. Therefore, in such scenarios, our D-NSS algorithm achieves an $O(m)$ improvement in sample efficiency over existing methods.

\section{Lower Bounds Under Heterogeneous Variance}
In the previous section, we established upper bounds for decentralized non-convex stochastic optimization under heterogeneous variance. 
Although the proposed algorithm achieves a sample complexity that depends on the arithmetic mean $\bar{\sigma}_{\mathrm{AM}}$, it remains unclear whether this dependence can be further reduced through more refined algorithmic designs. 
To address this, we investigate the lower bounds under heterogeneous variance. 
We begin by formally defining the class of algorithms.

\begin{definition}[Decentralized first-order algorithm class]
\label{def_3.1}
A decentralized first-order algorithm is defined over a network of nodes and satisfies the following constraints:
\begin{itemize}
\item \textbf{Local memory}: At time $t$, each node $i$ maintains a local memory $\mathcal{M}_{i,t} \subset \mathbb{R}^d$ that stores previously accessed or generated information. The memory is updated through either local computation or communication, i.e.,
\begin{align*}
\mathcal{M}_{i,t} \subseteq \mathcal{M}^{\text{comp}}_{i,t} \cup \mathcal{M}^{\text{comm}}_{i,t},
\end{align*}
where $\mathcal{M}^{\text{comp}}_{i,t}$ and $\mathcal{M}^{\text{comm}}_{i,t}$ represent the computational and communication memories, respectively.

\item \textbf{Local computation}: At time $t$, node $i$ can query a local first-order stochastic oracle to access $g_i(x; \xi_i)$ for any $x \in \mathcal{M}_{i,t-1}$. The computational memory is given by
\begin{align*}
\mathcal{M}^{\text{comp}}_{i,t} = \operatorname{Span} \left( \left\{ x,\, g_i(x; \xi_i) : x \in \mathcal{M}_{i,t-1} \right\} \right).
\end{align*}

\item \textbf{Local communication}: At time $t$, node $i$ can receive information from its neighbours $\mathcal{N}(i)$. The communication memory is defined as
\begin{align*}
\mathcal{M}^{\text{comm}}_{i,t} = \operatorname{Span} \left( \bigcup_{j \in \mathcal{N}(i)} \mathcal{M}_{j,t-\tau} \right),
\end{align*}
where $\tau < t$ denotes the communication delay parameter.

\item \textbf{Output value}: At time $t$, node $i$ must output one vector from its local memory as its current output, i.e., $x_i^t \in \mathcal{M}_{i,t}$.
\end{itemize}
\end{definition}

Establishing lower bounds for decentralized stochastic optimization under heterogeneous variance presents two key challenges.
First, existing results \cite{lu2021optimal, lu2023decentralized, yuan2022revisiting} rely on constructions in which all local objectives are identical, i.e., $f_1 = f_2 = \dots = f_m$. Such constructions are suitable for uniform-variance settings but fail to capture the difficulties introduced by heterogeneous variance. Motivated by lower bounds for finite-sum problems \cite{zhou2019lower}, we address this limitation by constructing orthogonal local functions that satisfy $\langle \nabla f_i(x), \nabla f_j(x) \rangle = 0$ for all $i \ne j$.
Second, heterogeneous variance breaks the symmetry across nodes, as algorithms typically allocate different numbers of samples to different nodes. This asymmetry introduces additional challenges in the analysis that are not addressed by existing lower bound techniques. Our construction explicitly incorporates this sample allocation asymmetry.
The detailed construction and proof are provided in the appendix.

\begin{table}[ht]
\centering
\caption{Comparison of sample complexity for first-order decentralized methods in the non-convex stochastic setting with heterogeneous variance.} \vskip0.2cm
\begin{tabular}{cc}
\hline
\textbf{Algorithm} & \textbf{Sample Complexity} \\
\hline\addlinespace
\makecell[c]{DeTAG \cite{lu2023decentralized}} &
$O\left( \dfrac{\Delta L \sigma_{\max}^2}{\epsilon^4} + \dfrac{m \Delta L}{\epsilon^2} \right)$ \\\addlinespace

\makecell[c]{MG-DSGD \cite{yuan2022revisiting}} &
$O\left( \dfrac{\Delta L \sigma_{\max}^2}{\epsilon^4} + \dfrac{m \Delta L}{\epsilon^2} \right)$ \\\addlinespace

\makecell[c]{GT-SA \cite{xin2021stochastic}} &
$O\left( \dfrac{\Delta L \bar{\sigma}_{\mathrm{QM}}^2}{\epsilon^4} + \dfrac{m \Delta L}{\epsilon^2} \right)$ \\\addlinespace

\makecell[c]{\textbf{D-NSS}\\\small Theorem~\ref{thm_2.1}} &
$O\left( \dfrac{\Delta L \bar{\sigma}_{\mathrm{AM}}^2}{\epsilon^4} + \dfrac{m \Delta L}{\epsilon^2} \right)$ \\\addlinespace

\hline\addlinespace
\makecell[c]{Lower Bound\\\small Theorem~\ref{thm_3.1}} &
$\Omega\left( \dfrac{\Delta L \bar{\sigma}_{\mathrm{AM}}^2}{\epsilon^4} + \dfrac{m \Delta L}{\epsilon^2} \right)$ \\\addlinespace
\hline
\end{tabular}
\label{table_1}
\end{table}

We state the main result as follows.
\begin{theorem}
\label{thm_3.1}
For any algorithm $\mathcal{A}$ satisfying Definition~\ref{def_3.1}, there exists a distributed objective function of the form $f(x) = \frac{1}{m} \sum_{i=1}^m f_i(x)$ with corresponding stochastic gradients satisfying Assumptions~\ref{ass_2.1}--\ref{ass_2.3}, such that, for sufficiently small $\epsilon$, the number of samples required to find an $\epsilon$-stationary point is at least
\begin{align*}
    \Omega\left( \frac{\Delta L \bar{\sigma}_{\mathrm{AM}}^2}{\epsilon^4} + \frac{m \Delta L}{\epsilon^2} \right).
\end{align*}
\end{theorem}

The result in Theorem~\ref{thm_3.1} shows that the D-NSS algorithm achieves the optimal sample complexity, and its dependence on the arithmetic mean $\bar{\sigma}_{\mathrm{AM}}$ is tight.
Moreover, since the construction of the communication lower bound does not involve stochastic gradients, the lower bound matches that in \citet{yuan2022revisiting} and is given by $\Omega (\epsilon^{-2} \Delta L/\sqrt{\chi})$.
Therefore, the D-NSS algorithm achieves nearly optimal communication complexity (up to a logarithmic factor).

\section{D-NSS-VR Algorithm}
In non-convex stochastic optimization, the optimal sample complexity is $O(\epsilon^{-4})$ under the standard bounded-variance assumption. When the mean-squared smoothness condition holds, this complexity can be improved to $O(\epsilon^{-3})$ \cite{arjevani2023lower}. A variety of variance-reduction methods, such as SARAH, SPIDER, PAGE, and their decentralized extensions, have been developed to attain this improved rate.

We have established a complexity bound that depends on the arithmetic mean $\bar{\sigma}_{\mathrm{AM}}$ in the general setting.
In contrast, decentralized variance-reduction methods involve more intricate structures, such as recursive gradient updates, nested inner–outer loops, and synchronization of local information, which require more delicate parameter tuning.
Whether a similar dependence on $\bar{\sigma}_{\mathrm{AM}}$ can be preserved within such a framework remains an open question.

In this section, we address this question by extending the sampling strategy to the variance-reduced setting. We propose D-NSS-VR (Decentralized Node-Specific Sampling with Variance Reduction), a decentralized algorithm that incorporates node-specific sampling into a SARAH-type variance reduction framework \cite{nguyen2017sarah}. We show that D-NSS-VR also achieves sample complexity that depends on the arithmetic mean $\bar{\sigma}_{\mathrm{AM}}$.

\subsection{Algorithm Design}

\begin{algorithm}[ht]
\caption{D-NSS-VR}
\label{algo:d-nss-vr}
\textbf{Input}: Initial point $x^0 \in \mathbb{R}^d$, step size $\eta > 0$, communication matrix $W$, communication rounds $R_t$ at iteration $t$, batch sizes $B_i$ and $b$, probability parameters $p, q \in (0,1]$ 
\begin{algorithmic}[1]
\STATE For all $i$, sample large batch $\mathcal{S}_i^0 = \{\xi_{i,1}, \dots, \xi_{i,B_i}\}$ \\
\STATE $y_i^0 = \dfrac{1}{B_i} \sum\limits_{j=1}^{B_i} g_i(x_i^0; \xi_{i,j})$; \\
\STATE $s_i^0 = \texttt{FastMix}(\{y_i^0\}_{i=1}^m, W, R_0)$; \\
\STATE $x_i^1 = \texttt{FastMix}(\{x_i^0 - \eta s_i^0\}_{i=1}^m, W, R_0)$
\FOR{$t = 1, 2, \dots, T - 1$}
  \STATE Sample $\zeta_t \sim \text{Bernoulli}(p)$
  \FOR{$i = 1, \dots, m$ \textbf{in parallel}}
    \IF{$\zeta_t = 1$}
      \STATE Sample large batch $\mathcal{S}_i^t = \{\xi_{i,1}, \dots, \xi_{i,B_i}\}$
      \STATE $y_i^t = \dfrac{1}{B_i} \sum\limits_{j=1}^{B_i} g_i(x_i^t; \xi_{i,j})$
    \ELSE
      \STATE Sample $\omega_i^t \sim \text{Bernoulli}(q)$
      \IF{$\omega_i^t = 1$}
        \STATE Sample mini-batch $\mathcal{S}_i^t = \{\xi_{i,1}, \dots, \xi_{i,b}\}$
        \STATE $y_i^t = y_i^{t-1} + \dfrac{\omega_i^t}{b q} \sum\limits_{j=1}^b \Big( g_i(x_i^t; \xi_{i,j})$ \\ $ - g_i(x_i^{t-1}; \xi_{i,j}) \Big)$
      \ELSE
        \STATE $y_i^t = y_i^{t-1}$
      \ENDIF
    \ENDIF
    \STATE $s_i^t = \texttt{FastMix}(\{s_i^{t-1} + y_i^t - y_i^{t-1}\}_{i=1}^m, W, R_t)$
    \STATE $x_i^{t+1} = \texttt{FastMix}(\{x_i^t - \eta s_i^t\}_{i=1}^m, W, R_t)$
  \ENDFOR
\ENDFOR
\end{algorithmic}
\textbf{Output}: $x_{i,\text{out}}$ uniformly sampled from $\{x_i^0, x_i^1, \dots, x_i^T\}$
\end{algorithm}

\noindent The core idea of variance reduction is to construct gradient estimates recursively using historical information and small mini-batches in most iterations, thereby significantly reducing sample usage while maintaining estimation accuracy.

As presented in Algorithm~\ref{algo:d-nss-vr}, D-NSS-VR integrates node-specific sampling with recursive gradient updates. During the large-batch update phase, the algorithm follows the design of D-NSS by allocating to each node $i$ a batch size $B_i$ proportional to its local noise level $\sigma_i$. This estimate serves as the initialization for subsequent recursive updates. In the inner update phase, a fixed mini-batch size $b$ is employed across all nodes to construct variance-reduced gradient estimates recursively, as specified in line 15 of Algorithm~\ref{algo:d-nss-vr}, which further improves the sample complexity.

Moreover, inspired by the PAGE method \cite{li2021page}, the algorithm employs a probabilistic update scheme, making the overall framework simpler and more unified from an analytical perspective. At the implementation level, D-NSS-VR also introduces a skip variable $w_i^t$, which allows each node to reuse the previous gradient estimate in certain iterations, thereby reducing computational overhead. In particular, as the noise level tends to zero, the algorithm reduces to a variance reduction method as in the finite-sum setting.

\subsection{Complexity Analysis of D-NSS-VR}

\begin{table*}[ht]
\centering
\caption{Comparison of sample and communication complexity for decentralized variance reduction methods under mean-squared smoothness and heterogeneous variance.}\vskip0.2cm
\begin{tabular}{ccc}
\hline
\textbf{Algorithm} & \textbf{Sample Complexity} & \textbf{Communication Rounds} \\
\hline\addlinespace
\makecell[c]{GT-HSGD \cite{xin2021hybrid}} & 
$O\left( \dfrac{(\Delta \bar{L} + \bar{\sigma}_{\mathrm{QM}}^2)^{3/2}}{\epsilon^3} \right)$ & 
$O\left( \dfrac{(\Delta \bar{L} + \bar{\sigma}_{\mathrm{QM}}^2)^{3/2}}{\epsilon^3} \right)$ \\\addlinespace

\makecell[c]{GT-SR-O \cite{xin2021stochastic}} & 
$O\left( \dfrac{\Delta \bar{L} \bar{\sigma}_{\mathrm{QM}}}{\epsilon^3} + \dfrac{\bar{\sigma}_{\mathrm{QM}}^2}{\epsilon^2} + \dfrac{m\Delta \bar{L}}{\epsilon^2} + m \right)$ & 
$\tilde{O}\left( \dfrac{\Delta \bar{L}}{\sqrt{\chi} \epsilon^2} + \dfrac{\bar{\sigma}_{\mathrm{QM}}}{\sqrt{\chi} \epsilon} \right)$ \\\addlinespace

\makecell[c]{\textbf{D-NSS-VR}\\\small Theorem~\ref{thm_4.1}} & 
$O\left( \dfrac{\Delta \bar{L} \bar{\sigma}_{\mathrm{AM}}}{\epsilon^3} + \dfrac{\bar{\sigma}_{\mathrm{AM}}^2}{\epsilon^2} + \dfrac{\sqrt{m} \Delta \bar{L}}{\epsilon^2} + m \right)$ & 
$\tilde{O}\left( \dfrac{\Delta \bar{L}}{\sqrt{\chi} \epsilon^2} + \dfrac{\bar{\sigma}_{\mathrm{AM}}}{\sqrt{\chi} \epsilon} \right)$ \\\addlinespace
\hline\addlinespace
\makecell[c]{Lower Bound\\\small Theorem~\ref{thm_4.2}} & 
$\Omega\left( \dfrac{\Delta \bar{L} \bar{\sigma}_{2/3}}{\epsilon^3} + \dfrac{\bar{\sigma}_{\mathrm{AM}}^2}{\epsilon^2} + \dfrac{\sqrt{m} \Delta \bar{L}}{\epsilon^2} + m \right)$ & 
$\Omega\left( \dfrac{\Delta \bar{L}}{\sqrt{\chi} \epsilon^2} \right)$ \\\addlinespace
\hline
\end{tabular}
\label{table_2}
\end{table*}

The convergence analysis of variance reduction methods relies on the following mean-squared smoothness assumption.
\begin{assumption}[Mean-Squared Smoothness]
\label{ass_2.5}
There exists a constant $\bar{L} > 0$ such that for any $x, y \in \mathbb{R}^d$,
\begin{align*}
\frac{1}{m} \sum_{i=1}^m \mathbb{E}_{\xi_i} \left\| g_i(x, \xi_i) - g_i(y, \xi_i) \right\|^2 \leq \bar{L}^2 \|x - y\|^2.
\end{align*}
\end{assumption}
Assumption~\ref{ass_2.5} is weaker than that in previous works \cite{pan2020d,sun2020improving,xin2021hybrid,xin2021stochastic}, which require mean-squared smoothness of the stochastic gradient for each local function. Moreover, this assumption also implies that the global objective function $f$ is $\bar{L}$-smooth. 

The following theorem provides upper bounds on the sampling and communication complexity of Algorithm~\ref{algo:d-nss-vr} for finding an $\epsilon$-stationary point.

\newpage
\begin{theorem}
\label{thm_4.1}
Under Assumptions~\ref{ass_2.1}, \ref{ass_2.2}, \ref{ass_2.4}, and \ref{ass_2.5}, consider Algorithm~\ref{algo:d-nss-vr} with the following parameter choices:
\begin{align*}
& \eta = \frac{1}{48\bar{L}}, \quad
B_i = \max\left\{ \left\lceil \frac{32 \sigma_i \sum_{j=1}^m \sigma_j}{m^2\epsilon^2} \right\rceil, 1 \right\}, \quad b = \left\lceil \frac{\sqrt{\sum_{i=1}^m B_i}}{m} \right\rceil,\quad
q = \frac{\sqrt{\sum_{i=1}^m B_i}}{b m},\\
& p = \frac{b q}{b q + \sum_{i=1}^m B_i / m},\quad T = \left\lceil \frac{384 \Delta \bar{L}}{\epsilon^2} + \frac{2}{p} \right\rceil, \quad R_0 = O\left( \frac{1}{\sqrt{\chi}} \log\left( \frac{m}{\epsilon} \right) \right), \;\text{and}\;
R_t = O\left( \frac{1}{\sqrt{\chi}} \log m \right),
\end{align*}
then the output $x_{i,\mathrm{out}}$ of the algorithm is an $\epsilon$-stationary point satisfying $\mathbb{E} \left[ \|\nabla f(x_{i,\mathrm{out}})\|^2 \right] \leq \epsilon^2$. The sample complexity is upper bounded by
\begin{align*}
O \left( \frac{\Delta \bar{L} \bar{\sigma}_{\mathrm{AM}}}{\epsilon^3} + \frac{\bar{\sigma}_{\mathrm{AM}}^2}{\epsilon^2} + \sqrt{m} \frac{ \Delta \bar{L}}{\epsilon^2} + m \right),
\end{align*}
and the communication complexity is bounded by
\begin{align*}
\tilde{O} \left( \frac{\Delta \bar{L}}{\sqrt{\chi} \epsilon^2} + \frac{\bar{\sigma}_{\text{AM}}}{\sqrt{\chi} \epsilon}\right).
\end{align*}
\end{theorem}

Theorem \ref{thm_4.1} shows that the variance reduction method can also achieve a dependence on the arithmetic mean of standard deviations.
As shown in Table~\ref{table_2}, we compare the sample and communication complexity of D-NSS-VR with state-of-the-art decentralized variance reduction methods \cite{xin2021hybrid,xin2021stochastic}. D-NSS-VR exhibits better dependence on the variance parameters, yielding an $O(\sqrt{m})$ improvement in sample complexity under strong heterogeneity. Moreover, when all variances $\sigma_i \to 0$, the sample complexity of D-NSS-VR matches that of variance reduction method for the finite-sum problem \cite{li2021page}, which is not achieved by the compared methods.

\subsection{Lower Bounds with Mean-Squared Smoothness}
We establish the following lower bound on the sample complexity for variance reduction algorithms.
\begin{theorem}
\label{thm_4.2}
For any algorithm $\mathcal{A}$ satisfying Definition~\ref{def_3.1}, there exists a distributed objective function of the form $f(x) = \frac{1}{m} \sum_{i=1}^m f_i(x)$ with corresponding stochastic gradients satisfying Assumptions~\ref{ass_2.1}, \ref{ass_2.2}, and \ref{ass_2.5}, such that, for sufficiently small $\epsilon$, the number of stochastic gradient samples required to find an $\epsilon$-stationary point is at least
\begin{align*}
    \Omega\left( \frac{\Delta \bar{L} \bar{\sigma}_{2/3}}{\epsilon^3} + \frac{\bar{\sigma}_{\mathrm{AM}}^2}{\epsilon^2} + \frac{\sqrt{m} \Delta \bar{L}}{\epsilon^2} + m \right),
\end{align*}
where $\bar{\sigma}_{2/3} = ((\sigma_1^{2/3} + \dots + \sigma_m^{2/3})/m)^{3/2}$.
\end{theorem}
Under mean-squared smoothness assumption, the communication lower bound remains $\Omega (\epsilon^{-2} \Delta L/\sqrt{\chi})$ as its construction does not rely on the stochastic gradients \cite{yuan2022revisiting}.

\begin{remark}
Theorems~\ref{thm_4.1} and \ref{thm_4.2} show that while several terms in the upper bounds match the corresponding lower bounds, a gap remains in the leading term. Due to the inherent difficulty of constructing lower bounds, particularly under the mean-squared smoothness assumption, closing this gap is left as an open problem for future work.
Although the upper and lower bounds do not fully match, our result demonstrates that the complexity of variance reduction methods can depend on $\bar{\sigma}_{\mathrm{AM}}$. Finally, the communication complexity is nearly optimal in the small-$\epsilon$ regime.
\end{remark}

\section{Numerical Experiments}

\begin{figure*}[t]
\centering
\begin{tabular}{ccc}
\includegraphics[scale=0.2]{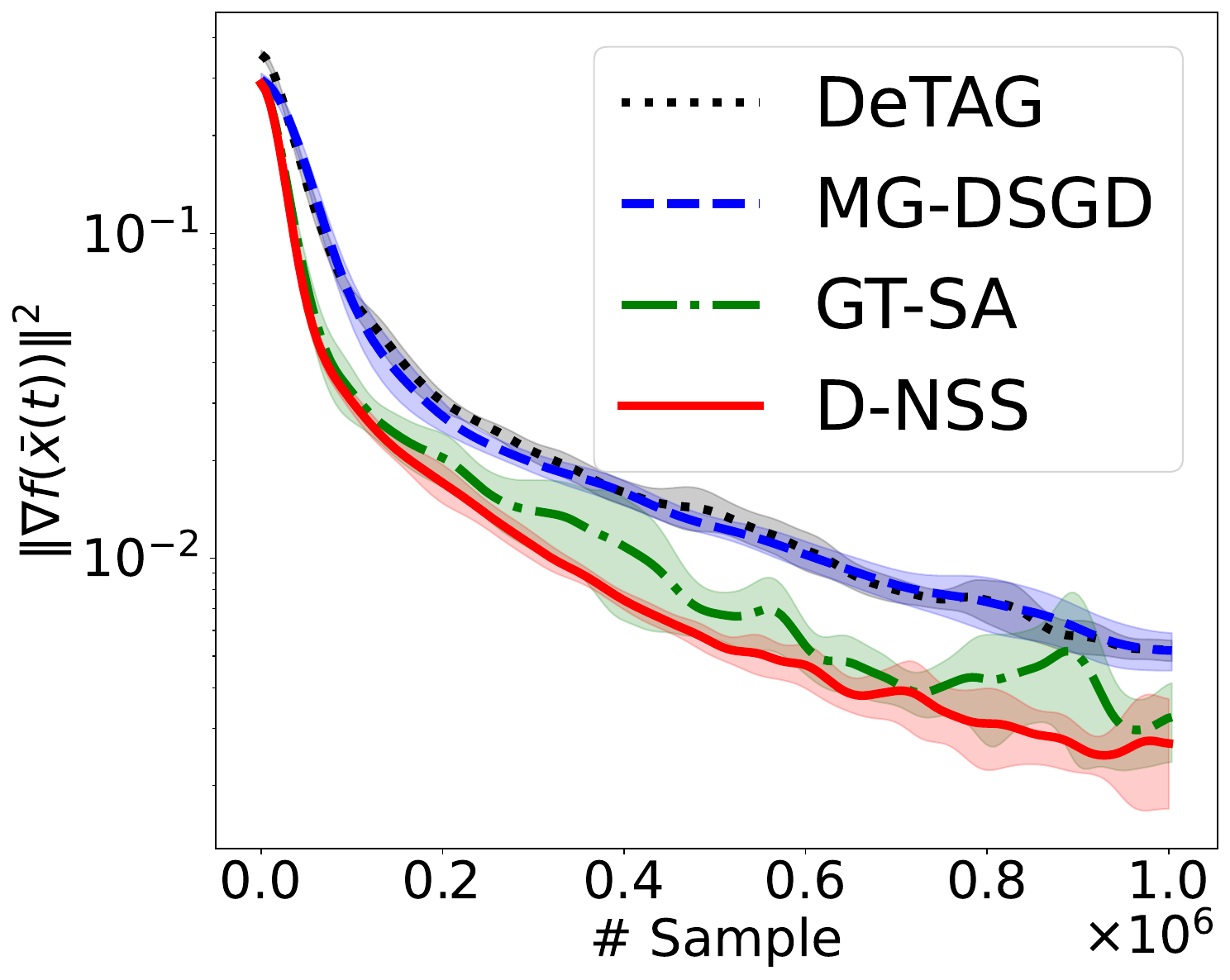} & 
\includegraphics[scale=0.2]{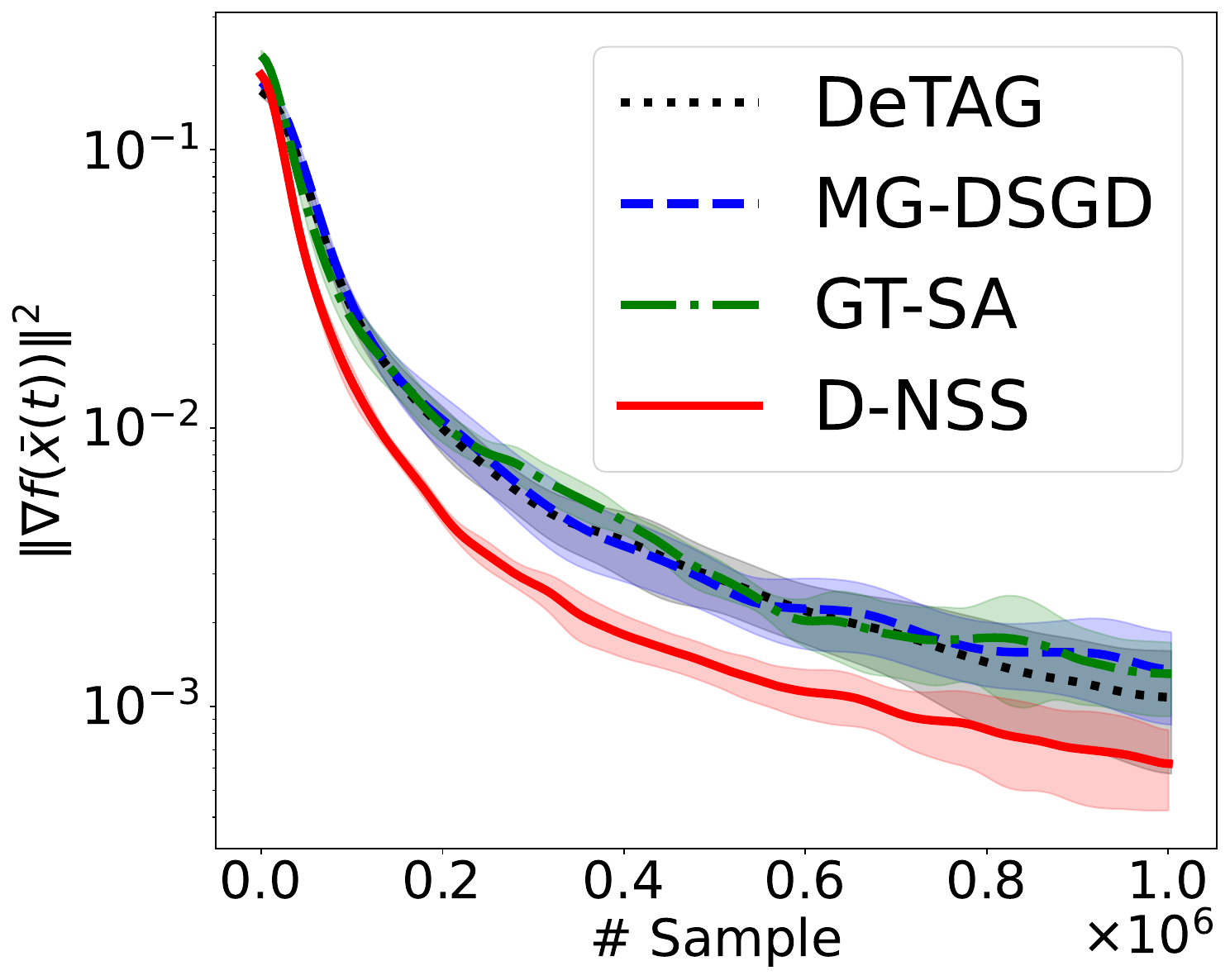}    & 
\includegraphics[scale=0.2]{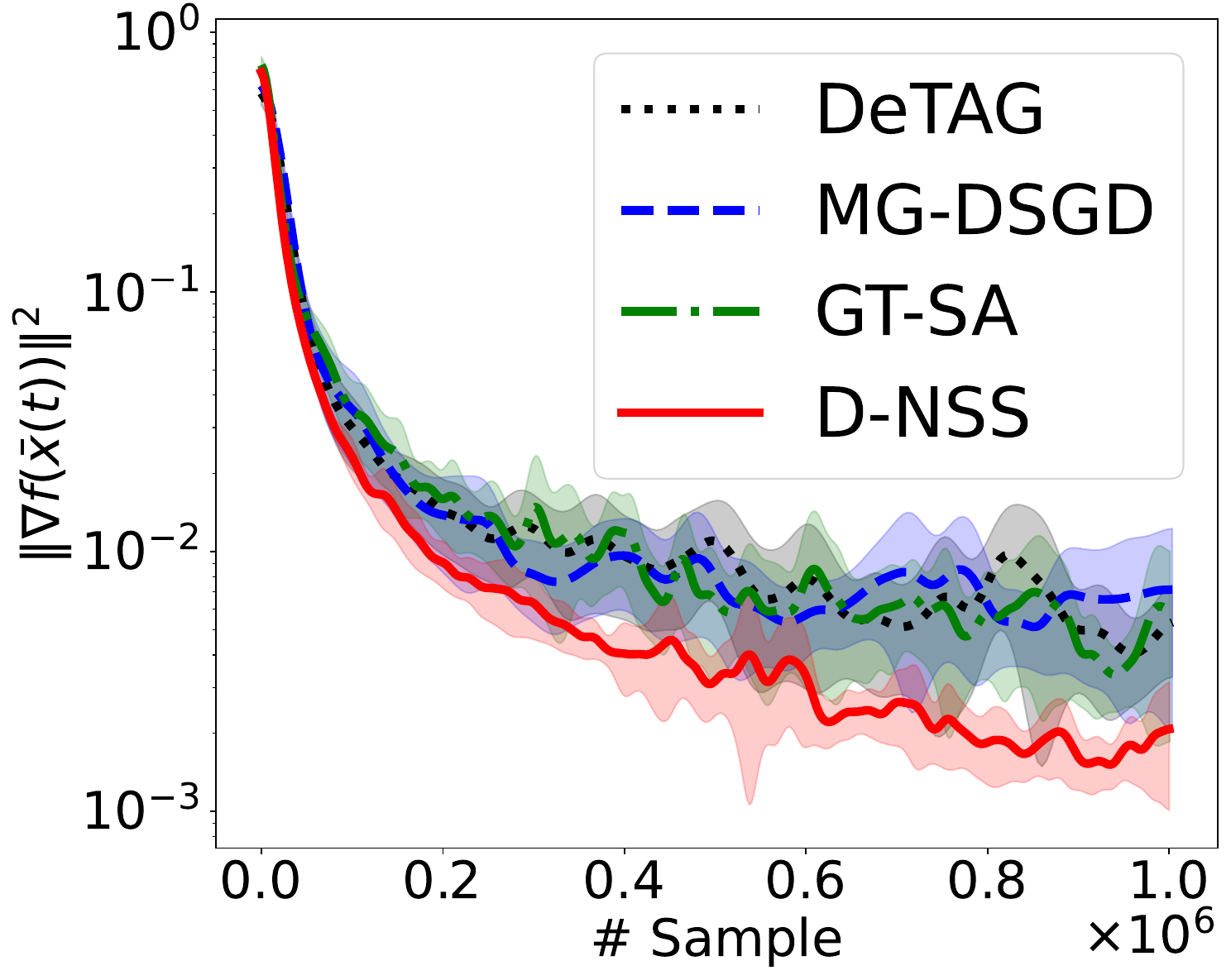} \\
(a) \texttt{a9a} & (b) \texttt{w8a} & (c) \texttt{MNIST}
\end{tabular} 
\caption{Performance comparison of decentralized algorithms in terms of the number of samples on datasets \texttt{a9a}, \texttt{w8a}, and \texttt{mnist}. The lines represent averages over 5 runs, and the shaded regions denote the standard deviations.}\label{fig1} 
\vskip0.2cm
\end{figure*}

\begin{figure*}[t]
\centering
\begin{tabular}{ccc}
\includegraphics[scale=0.2]{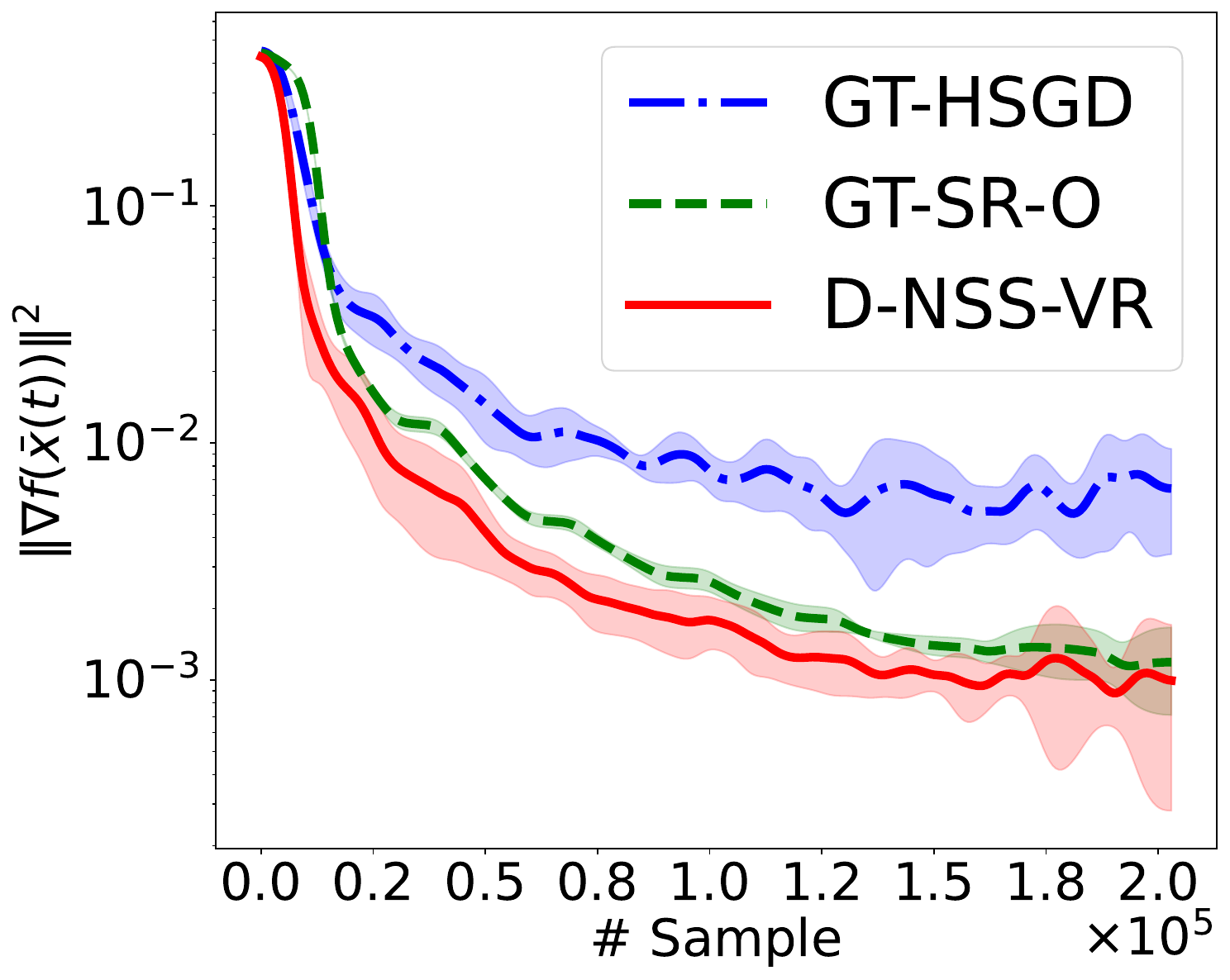} & 
\includegraphics[scale=0.2]{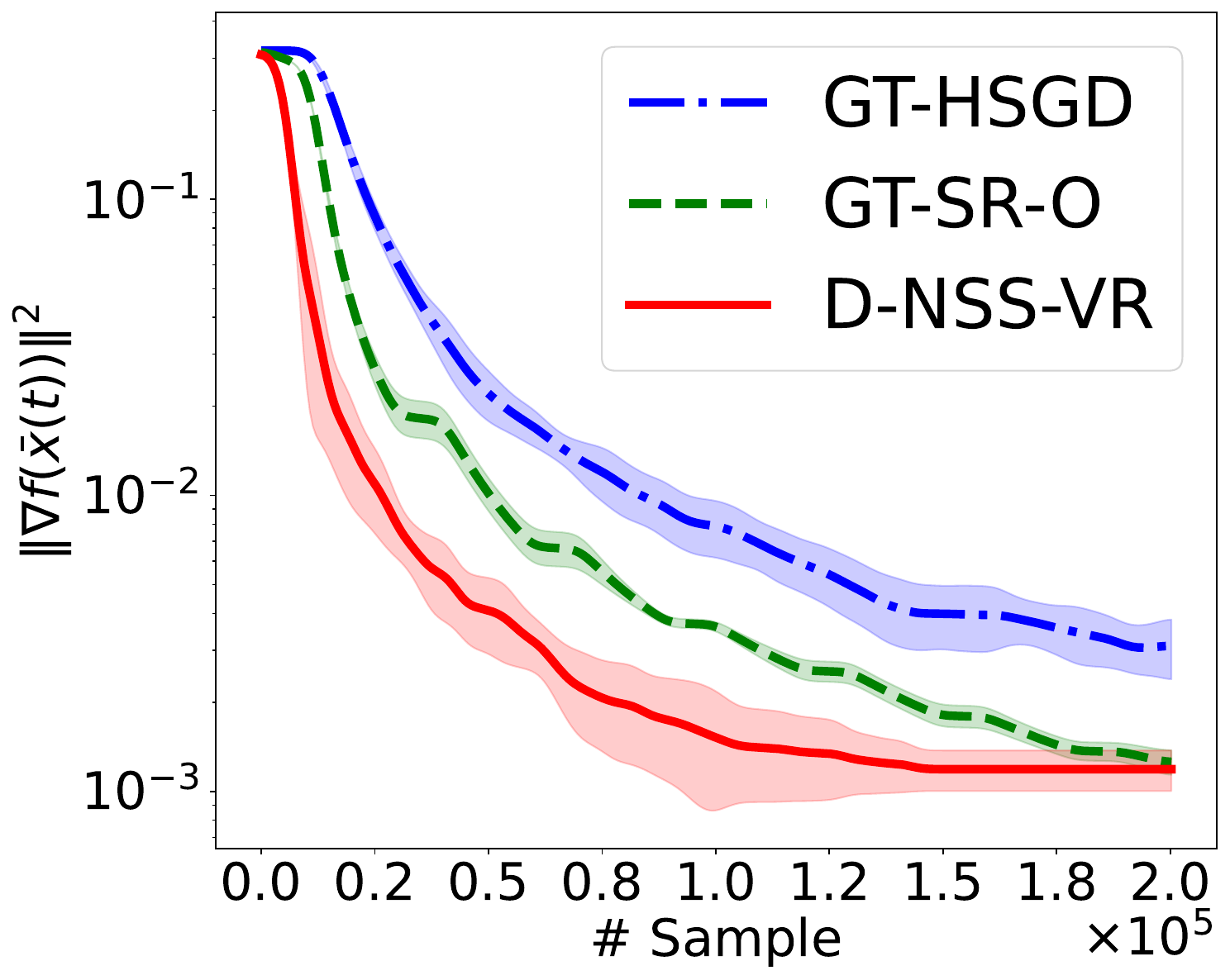}    & 
\includegraphics[scale=0.2]{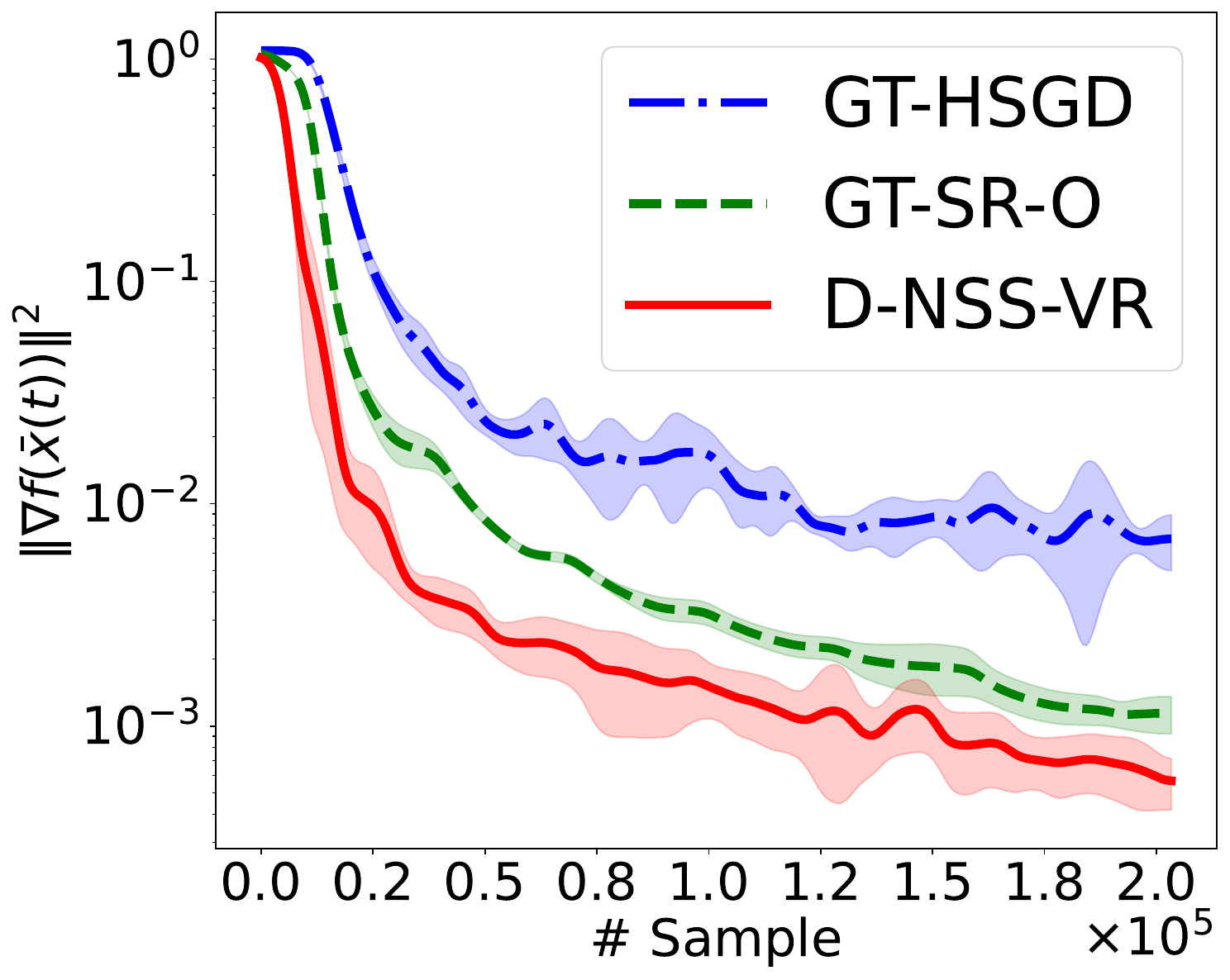} \\
(a) \texttt{a9a} & (b) \texttt{w8a} & (c) \texttt{MNIST}
\end{tabular} 
\caption{Performance comparison of decentralized variance reduction algorithms in terms of the number of samples on datasets \texttt{a9a}, \texttt{w8a}, and \texttt{mnist}. The lines represent averages over 5 runs, and the shaded regions denote the standard deviations.}\label{fig2} 
\end{figure*}

In this section, we validate the convergence of the proposed algorithms through numerical experiments on multiple datasets and compare them with existing methods. Specifically, we consider a widely used regularized logistic regression problem for binary classification in the decentralized setting~\cite{xin2021hybrid,luo2022on,gao2023distributed}:
\begin{equation*}
    f_i(x) = \frac{1}{N_i}\sum_{j=1}^{N_i} \log \left(1+\exp \left(-b_{ij} a_{ij}^{\top} x\right)\right)+ r \sum_{k=1}^d \frac{[x]_k^2}{1+[x]_k^2},
\end{equation*}
where $a_{ij} \in \mathbb{R}^d$ is the feature vector of the $j$-th sample at node $i$, $b_{ij} \in \{\pm 1\}$ is the corresponding binary label, and $r > 0$ is the regularization parameter, which is set to $10^{-4}$ in our experiments. Each node $i$ holds $N_i$ local samples, which may vary across nodes. To explicitly model heterogeneous variance across nodes, Gaussian noise with node-specific variances is added to the stochastic gradients.

We conduct numerical experiments on three real-world datasets: \texttt{a9a}, \texttt{w8a}, and \texttt{MNIST}. The datasets \texttt{a9a} ($N=32,561$, $d=123$) and \texttt{w8a} ($N=49,749$, $d=300$) can be downloaded from the LIBSVM repository \cite{chang2011libsvm}. For the \texttt{MNIST} dataset \cite{lecun2002gradient}, the digits 4 and 5 are utilized for the classification task, with $N=11,263$ and $d=784$.
We set the number of nodes $m = 20 $.
Regarding the communication network, a random graph with $\chi = 0.41$ is used and Metropolis-Hastings weights \cite{xiao2005scheme} are applied to construct the communication matrix $W$.

We first compare D-NSS (Algorithm~\ref{algo:d-nss}) with three baseline algorithms: DeTAG \cite{lu2023decentralized}, MG-DSGD \cite{yuan2022revisiting}, and GT-SA \cite{xin2021stochastic}. For each baseline, the batch sizes and step sizes are tuned to achieve their best empirical performance. In D-NSS, each $\sigma_i$ is estimated by computing the variance on a small batch (50 on each node) of stochastic gradients and broadcast at initialization. The additional cost of this variance estimation and communication is negligible compared with the overall cost.
The experimental results are presented in Figure~\ref{fig1}. It can be observed that the proposed D-NSS consistently outperforms the baseline methods across all three datasets, which demonstrates the effectiveness of the node-specific sampling strategy.

For the variance-reduced method D-NSS-VR (Algorithm~\ref{algo:d-nss-vr}), we compare it with the state-of-the-art variance-reduced algorithms GT-HSGD \cite{xin2021hybrid} and GT-SR-O \cite{xin2021stochastic}. The results are presented in Figure~\ref{fig2}. It can be observed that D-NSS-VR also consistently achieves the best performance across all three datasets. Although the incorporation of randomness (lines 6 and 12 in Algorithm~\ref{algo:d-nss-vr}) increases the variance, the combination with the node-specific sampling strategy leads to a significant advantage in practical performance.

\section{Conclusion}
This work studies decentralized non-convex stochastic optimization under heterogeneous variance. We propose D-NSS, a node-specific sampling algorithm that allocates samples according to local noise levels, and establish a sample complexity bound depending on the arithmetic mean of local standard deviations. A matching lower bound is derived, demonstrating the optimality of the proposed approach. Furthermore, by incorporating variance reduction, we develop D-NSS-VR, which achieves improved sample complexity under mean-squared smoothness while maintaining the arithmetic-mean dependence. Numerical experiments on multiple datasets validate our theoretical results and demonstrate the practical advantages of the proposed algorithms. Future work includes closing the remaining gap in the lower bound for variance reduction methods and investigating the performance of higher-order algorithms under heterogeneous variance.

\bibliographystyle{plainnat}
\bibliography{reference}

\newpage
\appendix
\section{Discussion on the Smoothness Assumptions}
\label{app_a}
In this section, we discuss the relationships between different smoothness assumptions.
We begin by recalling smoothness Assumption~\ref{ass_2.3}:
the global objective $f$ is $L$-smooth, and each local function $f_i$ is $mL$-smooth, i.e., for all $x, y \in \mathbb{R}^d$, we have
\begin{align}
\label{eq_smooth_g}
\|\nabla f(x) - \nabla f(y)\|^2 \leq L^2 \|x - y\|^2,
\end{align}
and for any node $i$ and $x, y \in \mathbb{R}^d$, we have
\begin{align}
\label{eq_smooth_ml}
\|\nabla f_i(x) - \nabla f_i(y)\|^2 \leq m^2 L^2 \|x - y\|^2.
\end{align}
In contrast, existing works \cite{lu2021optimal, lu2023decentralized, yuan2022revisiting} commonly assume that each local function $f_i$ is $L$‑smooth, i.e.,
\begin{align}
\label{eq_smooth_l}
\|\nabla f_i(x) - \nabla f_i(y)\|^2 \leq L^2 \|x-y\|^2.
\end{align}
We then show that the assumption in \eqref{eq_smooth_l} implies \eqref{eq_smooth_g} and \eqref{eq_smooth_ml}. 
Noting that
\begin{align*}
\|\nabla f(x) - \nabla f(y)\|^2 &= \left\|\frac{1}{m}\sum_{i=1}^m \left( \nabla f_i(x) -\nabla f_i(y) \right)\right\|^2 \\
& \leq \frac{1}{m}\sum_{i=1}^m \|\nabla f_i(x) - \nabla f_i(y)\|^2 \\
& \leq \frac{1}{m}\sum_{i=1}^m L^2 \|x-y\|^2 \\
& = L^2 \|x-y\|^2,
\end{align*}
and 
\begin{align*}
\|\nabla f_i(x) - \nabla f_i(y)\|^2 \leq L^2 \|x-y\|^2 \leq m^2L^2 \|x-y\|^2.
\end{align*}
Therefore, Assumption~\ref{ass_2.3} is weaker than the smoothness condition in existing works \cite{lu2021optimal, lu2023decentralized, yuan2022revisiting}.

For the mean-square smoothness Assumption~\ref{ass_2.5}, there exists a constant $\bar{L} > 0$ such that for any $x, y \in \mathbb{R}^d$,
\begin{align}
\label{eq_smooth_ms}
\frac{1}{m} \sum_{i=1}^m \mathbb{E}_{\xi_i} \left\| g_i(x, \xi_i) - g_i(y, \xi_i) \right\|^2 \leq \bar{L}^2 \|x - y\|^2.
\end{align}
In contrast, existing variance reduction methods \cite{pan2020d,sun2020improving,xin2021hybrid,xin2021stochastic} commonly assume that
\begin{align}
\label{eq_smooth_lms}
\mathbb{E}_{\xi_i} \left\| g_i(x, \xi_i) - g_i(y, \xi_i) \right\|^2 \leq \bar{L}^2 \|x - y\|^2.
\end{align}
Note that if \eqref{eq_smooth_lms} holds, then
\begin{align*}
\frac{1}{m} \sum_{i=1}^m \mathbb{E}_{\xi_i} \left\| g_i(x, \xi_i) - g_i(y, \xi_i) \right\|^2
\leq \frac{1}{m} \sum_{i=1}^m \bar{L}^2 \|x - y\|^2
= \bar{L}^2 \|x - y\|^2,
\end{align*}
which implies \eqref{eq_smooth_ms}. Therefore, Assumption~\ref{ass_2.5} is weaker.

Finally, we show that Assumption~\ref{ass_2.5} implies that the global objective function $f$ is $\bar{L}$‑smooth. 
For any $x, y \in \mathbb{R}^d$, we have
\begin{align*}
    \|\nabla f(x) - \nabla f(y)\|^2 & = \left\| \frac{1}{m} \sum_{i=1}^m \mathbb{E}_{\xi_i} \left[ g_i(x;\xi_i) - g_i(y;\xi_i) \right]\right\|^2\\
    & \leq \frac{1}{m} \sum_{i=1}^m \left\| \mathbb{E}_{\xi_i}\left[ g_i(x, \xi_i) - g_i(y, \xi_i) \right] \right\|^2 \\
    & \leq \frac{1}{m} \sum_{i=1}^m \mathbb{E}_{\xi_i} \left\| g_i(x, \xi_i) - g_i(y, \xi_i) \right\|^2\\
    &\leq \bar{L}^2 \|x - y\|^2,
\end{align*}
where the second inequality follows from Jensen’s inequality, and the last inequality holds by the mean‑square smoothness assumption.

\section{Some Basic Results}

In this section, we present several basic results that will be used in the subsequent proofs.
First, we recall the notation used in the paper. The variable $x_i^t$ denotes the local variable at node $i$ in iteration $t$, $y_i^t$ denotes the gradient estimate, and $s_i^t$ denotes the gradient tracking variable. We define the averaged variables as $\bar{x}^t := \frac{1}{m} \sum_{i=1}^m x_i^t$, $\bar{y}^t := \frac{1}{m} \sum_{i=1}^m y_i^t$, and $\bar{s}^t := \frac{1}{m} \sum_{i=1}^m s_i^t$.
Uppercase letters are used to denote matrices stacking the corresponding vectors, e.g.,  
\begin{align*}
X^t =
\begin{bmatrix}
x_1^t \\
\vdots \\
x_m^t
\end{bmatrix}, \quad
Y^t =
\begin{bmatrix}
y_1^t \\
\vdots \\
y_m^t
\end{bmatrix},\quad \text{and} \quad
S^t =
\begin{bmatrix}
s_1^t \\
\vdots \\
s_m^t
\end{bmatrix}
\in \mathbb{R}^{m \times d}.
\end{align*}
For $\bar{x} \in \mathbb{R}^d$, define $\mathbf{1}\bar{x} \in \mathbb{R}^{m\times d}$ as the matrix whose rows are all equal to $\bar{x}$. The notation $\nabla \mathbf{f}(X^t) \in \mathbb{R}^{m \times d}$ refers to the matrix whose $i$‑th row is $\nabla f_i(x_i^t) \in \mathbb{R}^d$.
The notation $\| \cdot \|$ represents the Euclidean norm of a vector or the Frobenius norm of a matrix.

Then we present some useful lemmas.
\begin{lemma}
Under Assumption~\ref{ass_2.5}, for any $x,y \in \mathbb{R}^d$, it holds that
\begin{align*}
\mathbb{E}_{\xi_i} \left[ \| g_i(x;\xi_i) - g_i(y;\xi_i) \|^2 \right] \le m \bar{L}^2 \|x-y \|^2.
\end{align*}
\end{lemma}
\begin{proof}
By Assumption~\ref{ass_2.5}, we have
\begin{align*}
\mathbb{E}_{\xi_i} \left[ \| g_i(x;\xi_i) - g_i(y;\xi_i) \|^2 \right] \le \sum_{i=1}^m \mathbb{E}_{\xi_i} \left[ \| g_i(x;\xi_i) - g_i(y;\xi_i) \|^2 \right] \le m \bar{L}^2 \|x-y \|^2,
\end{align*}
which completes the proof.
\end{proof}

The following proposition \cite{ye2023multi} characterizes the convergence of Algorithm~\ref{algo:fastmix}.
\begin{lemma}[{\citet[Proposition 1]{ye2023multi}}]
\label{lemma_1.1}
    Under Assumption~\ref{ass_2.4}, the output $Z^R = \mathtt{FastMix}(\{z^0_i\}_{i=1}^m, W, R)$ of Algorithm~\ref{algo:fastmix} holds that
    \begin{align*}
    \frac{1}{m} \mathbf{1}^\top Z^{R} = \bar{z}^0
    \end{align*}
    and
    \begin{align*}        
    \| Z^R - \mathbf{1} \bar{z}^0 \| \leq c_1 \big(1 - c_2 \sqrt{1 - \lambda_2(W)}\,\big)^R \| Z^0 - \mathbf{1} \bar{z}^0 \|,
    \end{align*}
    where $\bar{z}^0 = \frac{1}{m} \mathbf{1}^\top Z^0, c_1 = \sqrt{14}$, and $c_2 = 1 - 1/\sqrt{2}.$
\end{lemma}
To characterize the convergence of the multi-consensus steps, define
\begin{align*}
\rho_t := \sqrt{14} (1 - (1 - 1/\sqrt{2}) \sqrt{1 - \lambda_2(W)})^{R_t}.
\end{align*}
Then we have 
\begin{align*}        
    \| \mathtt{FastMix}(A, W, R_t) - \frac{1}{m}\mathbf{1}\mathbf{1}^\top \mathtt{FastMix}(A, W, R_t) \| \leq \rho_t \| A - \frac{1}{m}\mathbf{1}\mathbf{1}^\top A \|
\end{align*}
for any $A \in \mathbb{R}^{m\times d}$.

The following lemma shows a crucial property of gradient tracking.
\begin{lemma}
    For Algorithms 1 and 3, it holds that 
    \begin{align}
    \label{eq_tracking}
    \bar{s}^t = \bar{y}^t.
    \end{align} 
\end{lemma}
\begin{proof}
From Lemma~\ref{lemma_1.1}, combined with the update of $S^t$ in Algorithms 1 and 3, we have
\begin{align*}
    \bar{s}^t & = \frac{1}{m} \mathbf{1}^\top S^t \\
    & = \frac{1}{m} \mathbf{1}^\top \mathtt{FastMix}(S^{t-1} + Y^t-Y^{t-1}, W, R) \\
    & = \bar{s}^{t-1} + \bar{y}^t - \bar{y}^{t-1}.
\end{align*}
Note that $\bar{s}^0 = \bar{y}^0$. 
By induction, it is straightforward to verify that $\bar{s}^t = \bar{y}^t$ holds for $t\ge 0$.
\end{proof}

\section{The Proof of Theorem~\ref{thm_2.1}}
We first define the Lyapunov function
\begin{align*}
\Phi(t) := f(\bar{x}^t) + 2 m L^2 \eta C(t),
\end{align*}
where
$C(t) := \left\| X^t - \mathbf{1} \bar{x}^t \right\|^2 + \eta^2 \left\| S^t - \mathbf{1} \bar{s}^t \right\|^2$
represents the consensus error term.  We then derive upper bounds for $f(\bar{x}^t)$ and $C(t)$ in the Lyapunov function.

By the $L$‑smoothness of $f$, the following inequality holds for any $t \ge 0$:
\begin{align}
\label{eq_2.4}
f(\bar{x}^{t+1}) 
&\le f(\bar{x}^t) + \langle \nabla f(\bar{x}^t), \bar{x}^{t+1} - \bar{x}^t \rangle + \frac{L}{2} \| \bar{x}^{t+1} - \bar{x}^t \|^2 \notag\\
&= f(\bar{x}^t) - \langle \nabla f(\bar{x}^t), \eta \bar{y}^t \rangle + \frac{L \eta^2}{2} \| \bar{y}^t \|^2 \notag\\
&= f(\bar{x}^t) + \frac{\eta}{2} \left( \| \bar{y}^t - \nabla f(\bar{x}^t) \|^2 - \| \nabla f(\bar{x}^t) \|^2 - \| \bar{y}^t \|^2 \right) + \frac{L \eta^2}{2} \| \bar{y}^t \|^2 \notag\\
&= f(\bar{x}^t) - \frac{\eta}{2} \| \nabla f(\bar{x}^t) \|^2 - \left( \frac{\eta}{2} - \frac{L \eta^2}{2} \right) \| \bar{y}^t \|^2 + \frac{\eta}{2} \| \bar{y}^t - \nabla f(\bar{x}^t) \|^2.
\end{align}
Consider the gradient estimation error
$\mathbb{E}\left[\left\| \bar{y}^t - \nabla f(\bar{x}^t) \right\|^2\right]$
in the last term of \eqref{eq_2.4}. By the definition of $\bar{y}^t$, it follows that
\begin{align}
\label{eq_2.5}
\mathbb{E} \left[ \left\| \bar{y}^t - \nabla f(\bar{x}^t) \right\|^2 \right] 
&= \mathbb{E} \left[ \left\| \frac{1}{m} \sum_{i=1}^m \frac{1}{B_i} \sum_{j=1}^{B_i} g_i(x_i^t; \xi_{ij}) - \nabla f(\bar{x}^t) \right\|^2 \right] \notag\\
&\le 2 \, \mathbb{E} \left[ \left\| \frac{1}{m} \sum_{i=1}^m \left( \frac{1}{B_i} \sum_{j=1}^{B_i} g_i(x_i^t; \xi_{ij}) - \nabla f_i(x_i^t) \right) \right\|^2 \right] \notag\\
&\quad + 2 \, \mathbb{E} \left[ \left\| \frac{1}{m} \sum_{i=1}^m \left( \nabla f_i(x_i^t) - \nabla f_i(\bar{x}^t) \right) \right\|^2 \right] \notag\\
&= \frac{2}{m^2} \sum_{i=1}^m \frac{1}{B_i^2} \sum_{j=1}^{B_i} \mathbb{E} \left[ \left\| g_i(x_i^t; \xi_{ij}) - \nabla f_i(x_i^t) \right\|^2 \right] \notag\\
&\quad + 2 \, \mathbb{E} \left[ \left\| \frac{1}{m} \sum_{i=1}^m  \left( \nabla f_i(x_i^t) - \nabla f_i(\bar{x}^t) \right) \right\|^2 \right] \notag\\
&\le \frac{2}{m^2} \sum_{i=1}^m \frac{\sigma_i^2}{B_i} + \frac{2}{m} \sum_{i=1}^m \mathbb{E} \left[ \left\| \nabla f_i(x_i^t) - \nabla f_i(\bar{x}^t) \right\|^2 \right] \notag\\
&\le \frac{2}{m^2} \sum_{i=1}^m \frac{\sigma_i^2}{B_i} + 2mL^2 \, \mathbb{E} \left[ \left\| X^t - \mathbf{1} \bar{x}^t \right\|^2 \right],
\end{align}
where the second inequality follows from Assumption~\ref{ass_2.2} and the last inequality follows from Assumption~\ref{ass_2.3}.

We next establish the upper bound for the consensus error term $\| X^t - \mathbf{1} \bar{x}^t \|^2$.
Using the convergence property of the $\mathtt{FastMix}$ algorithm stated in Lemma \ref{lemma_1.1}, define
$\rho_t := \sqrt{14} (1 - (1 - 1/\sqrt{2}) \sqrt{1 - \lambda_2(W)})^{R_t}$.
Then 
\begin{align}
\label{eq_2.6}
& \quad \left\| X^{t+1} - \mathbf{1} \bar{x}^{t+1} \right\|^2 \notag\\
&= \left\| \mathtt{FastMix} \left( X^t - \eta S^t , W, R_t\right) - \frac{1}{m} \mathbf{1} \mathbf{1}^\top \mathtt{FastMix} \left( X^t - \eta S^t, W, R_t\right) \right\|^2 \notag\\
&\le \rho_t^2 \left\| X^t - \eta S^t - \frac{1}{m} \mathbf{1} \mathbf{1}^\top \left( X^t - \eta S^t \right) \right\|^2 \notag\\
&= \rho_t^2 \left\| X^t - \mathbf{1} \bar{x}^t - \eta \left( S^t - \mathbf{1} \bar{s}^t \right) \right\|^2 \notag\\
&\le 2 \rho_t^2 \left\| X^t - \mathbf{1} \bar{x}^t \right\|^2 + 2 \rho_t^2 \eta^2 \left\| S^t - \mathbf{1} \bar{s}^t \right\|^2,
\end{align}
where the first inequality follows from Lemma \ref{lemma_1.1}, the second from the fact that
$\|A - \tfrac{1}{m}\mathbf{1}\mathbf{1}^\top A\| \le \|A\|$ for any $A \in \mathbb{R}^{m\times d}$,
and the last inequality from Young’s inequality.

Furthermore, we bound the consensus error term of the gradient tracking variables
$\left\| S^t - \mathbf{1} \bar{s}^t \right\|^2$
in a similar manner:
\begin{align}
\label{eq_2.7}
& \quad \left\| S^{t+1} - \mathbf{1} \bar{s}^{t+1} \right\|^2 \notag\\
&= \left\| \mathtt{FastMix} \left( S^t + Y^{t+1} - Y^t, W, R_t \right) - \frac{1}{m} \mathbf{1} \mathbf{1}^\top \mathtt{FastMix} \left( S^t + Y^{t+1} - Y^t, W, R_t \right) \right\|^2 \notag\\
&\le \rho_t^2 \left\| S^t + Y^{t+1} - Y^t - \frac{1}{m} \mathbf{1} \mathbf{1}^\top \left( S^t + Y^{t+1} - Y^t \right) \right\|^2 \notag\\
&\le 2 \rho_t^2 \left\| S^t - \mathbf{1} \bar{s}^t \right\|^2 + 2 \rho_t^2 \left\| Y^{t+1} - Y^t - \frac{1}{m} \mathbf{1} \mathbf{1}^\top \left( Y^{t+1} - Y^t \right) \right\|^2 \notag\\
&\le 2 \rho_t^2 \left\| S^t - \mathbf{1} \bar{s}^t \right\|^2 + 2 \rho_t^2 \left\| Y^{t+1} - Y^t \right\|^2.
\end{align}
The last term $\mathbb{E} \left[ \left\| Y^{t+1} - Y^t \right\|^2 \right]$ of  \eqref{eq_2.7} has the upper bound
\begin{align}
\label{eq_2.8}
&\quad \mathbb{E} \left[ \left\| Y^{t+1} - Y^t \right\|^2 \right] \notag\\
&= \sum_{i=1}^m \mathbb{E} \left[ \left\| y_i^{t+1} - y_i^t \right\|^2 \right] \notag\\
&= \sum_{i=1}^m \mathbb{E} \left[ \left\| y_i^{t+1} - \nabla f_i(x_i^{t+1}) + \nabla f_i(x_i^{t+1}) - \nabla f_i(x_i^t) + \nabla f_i(x_i^t) - y_i^t \right\|^2 \right] \notag\\
&\le 3 \sum_{i=1}^m \mathbb{E} \left[ \left\| y_i^{t+1} - \nabla f_i(x_i^{t+1}) \right\|^2 \right]
+ 3 \sum_{i=1}^m \mathbb{E} \left[ \left\| \nabla f_i(x_i^{t+1}) - \nabla f_i(x_i^t) \right\|^2 \right] \notag\\
&\quad + 3 \sum_{i=1}^m \mathbb{E} \left[ \left\| y_i^t - \nabla f_i(x_i^t) \right\|^2 \right] \notag\\
&\le 6 \sum_{i=1}^m \frac{\sigma_i^2}{B_i}
+ 3 \sum_{i=1}^m m^2 L^2 \mathbb{E} \left[ \left\| x_i^{t+1} - x_i^t \right\|^2 \right] \notag\\
&\le 6 \sum_{i=1}^m \frac{\sigma_i^2}{B_i}
+ 9 m^2 L^2 \mathbb{E} \left[ \left\| X^{t+1} - \mathbf{1} \bar{x}^{t+1} \right\|^2 \right]
+ 9 m^2 L^2 \mathbb{E} \left[ \left\| X^t - \mathbf{1} \bar{x}^t \right\|^2 \right] \notag\\
&\quad + 9 m^3 L^2 \mathbb{E} \left[ \left\| \bar{x}^{t+1} - \bar{x}^t \right\|^2 \right] \notag\\
&\le 6 \sum_{i=1}^m \frac{\sigma_i^2}{B_i}
+ \left( 18 m^2 L^2 \rho_t^2 + 9 m^2 L^2 \right) \mathbb{E} \left[ \left\| X^t - \mathbf{1} \bar{x}^t \right\|^2 \right] \\
&\quad + 18 m^2 L^2 \rho_t^2 \eta^2 \mathbb{E} \left[ \left\| S^t - \mathbf{1} \bar{s}^t \right\|^2 \right]
+ 9 m^3 L^2 \eta^2 \mathbb{E} \left[ \left\| \bar{y}^t \right\|^2 \right]. \notag
\end{align}
Combining the results of \eqref{eq_2.7} and \eqref{eq_2.8}, we obtain
\begin{align}
\label{eq_2.9}
\mathbb{E} \left[ \left\| S^{t+1} - \mathbf{1} \bar{s}^{t+1} \right\|^2 \right]
&\le 2 \rho_t^2 \, \mathbb{E} \left[ \left\| S^t - \mathbf{1} \bar{s}^t \right\|^2 \right] 
+ 2 \rho_t^2 \, \mathbb{E} \left[ \left\| Y^{t+1} - Y^t \right\|^2 \right] \notag\\
&\le \left( 2 \rho_t^2 + 36 m^2 L^2 \eta^2 \rho_t^4 \right) \mathbb{E} \left[ \left\| S^t - \mathbf{1} \bar{s}^t \right\|^2 \right] \\
&\quad + \left( 36 m^2 L^2 \rho_t^4 + 18 m^2 L^2 \rho_t^2 \right) \mathbb{E} \left[ \left\| X^t - \mathbf{1} \bar{x}^t \right\|^2 \right] \notag\\
&\quad + 18 m^3 L^2 \eta^2 \rho_t^2 \, \mathbb{E} \left[ \left\| \bar{y}^t \right\|^2 \right]
+ 12 \rho_t^2 \sum_{i=1}^m \frac{\sigma_i^2}{B_i}.\notag
\end{align}

From \eqref{eq_2.6} and \eqref{eq_2.9}, the following relation holds for $C(t)$ in the Lyapunov function
\begin{align}
\label{eq_2.10}
\mathbb{E} \left[ C(t+1) \right]
&= \mathbb{E} \left[ \left\| X^{t+1} - \mathbf{1} \bar{x}^{t+1} \right\|^2 \right] 
+ \eta^2 \, \mathbb{E} \left[ \left\| S^{t+1} - \mathbf{1} \bar{s}^{t+1} \right\|^2 \right] \notag\\
&\le 2 \rho_t^2 \, \mathbb{E} \left[ \left\| X^t - \mathbf{1} \bar{x}^t \right\|^2 \right]
+ 2 \rho_t^2 \eta^2 \, \mathbb{E} \left[ \left\| S^t - \mathbf{1} \bar{s}^t \right\|^2 \right] \notag\\
&\quad + \left( 2 \rho_t^2 + 36 m^2 L^2 \eta^2 \rho_t^4 \right) \eta^2 \, \mathbb{E} \left[ \left\| S^t - \mathbf{1} \bar{s}^t \right\|^2 \right] \notag\\
&\quad + \left( 36 m^2 L^2 \rho_t^4 + 18 m^2 L^2 \rho_t^2 \right) \eta^2 \, \mathbb{E} \left[ \left\| X^t - \mathbf{1} \bar{x}^t \right\|^2 \right] \notag\\
&\quad + 18 m^3 L^2 \eta^4 \rho_t^2 \, \mathbb{E} \left[ \left\| \bar{y}^t \right\|^2 \right]
+ 12 \rho_t^2 \eta^2 \sum_{i=1}^m \frac{\sigma_i^2}{B_i} \notag\\
&\le 2 \rho_t^2 \left( 1 + 27 m^2 L^2 \eta^2 \right) \mathbb{E} \left[ \left\| X^t - \mathbf{1} \bar{x}^t \right\|^2 \right] \notag\\
&\quad + 2 \rho_t^2 \eta^2 \left( 2 + 18 m^2 L^2 \eta^2 \right) \mathbb{E} \left[ \left\| S^t - \mathbf{1} \bar{s}^t \right\|^2 \right] \notag\\
&\quad + 18 m^3 L^2 \eta^4 \rho_t^2 \, \mathbb{E} \left[ \left\| \bar{y}^t \right\|^2 \right]
+ 12 \rho_t^2 \eta^2 \sum_{i=1}^m \frac{\sigma_i^2}{B_i} \notag\\
&= 2 \rho_t^2 \left( 27 m^2 L^2 \eta^2 + 2 \right) \mathbb{E} \left[ C(t)\right] 
+ 18 m^3 L^2 \eta^4 \rho_t^2 \, \mathbb{E} \left[ \left\| \bar{y}^t \right\|^2 \right] \\
& \quad + 12 \rho_t^2 \eta^2 \sum_{i=1}^m \frac{\sigma_i^2}{B_i}. \notag
\end{align}

Combining the results of \eqref{eq_2.4}, \eqref{eq_2.5}, and \eqref{eq_2.10}, we obtain the recursive relation of the Lyapunov function $\Phi(t)$:
\begin{align*}
\mathbb{E} \left[ \Phi(t+1) \right] 
&= \mathbb{E} \left[ f(\bar{x}^{t+1}) \right] + 2 m L^2 \eta \, \mathbb{E} \left[ C(t+1) \right] \\
&\le \mathbb{E} \left[ f(\bar{x}^t) \right] - \frac{\eta}{2} \mathbb{E} \left[ \left\| \nabla f(\bar{x}^t) \right\|^2 \right]
- \left( \frac{\eta}{2} - \frac{L \eta^2}{2} \right) \mathbb{E} \left[ \left\| \bar{y}^t \right\|^2 \right] \\
&\quad + \frac{\eta}{m^2} \sum_{i=1}^m \frac{\sigma_i^2}{B_i} + \eta m L^2 \mathbb{E} \left[ C(t) \right] \\
&\quad + 4 m L^2 \eta \rho_t^2 \left( 27 m^2 L^2 \eta^2 + 2 \right) \mathbb{E} \left[ C(t) \right] 
+ 36 m^4 L^4 \eta^5 \rho_t^2 \mathbb{E} \left[ \left\| \bar{y}^t \right\|^2 \right] \\
&\quad + 24 m L^2 \eta^3 \rho_t^2 \sum_{i=1}^m \frac{\sigma_i^2}{B_i} \\
&= \mathbb{E} \left[ f(\bar{x}^t) \right] - \frac{\eta}{2} \mathbb{E} \left[ \left\| \nabla f(\bar{x}^t) \right\|^2 \right]
- \left( \frac{\eta}{2} - \frac{L \eta^2}{2} - 36 m^4 L^4 \eta^5 \rho_t^2 \right) \mathbb{E} \left[ \left\| \bar{y}^t \right\|^2 \right] \\
&\quad + \left( 4 m L^2 \eta \left( 27 m^2 L^2 \eta^2 + 2 \right) \rho_t^2 + m L^2 \eta \right) \mathbb{E} \left[ C(t) \right] \notag\\
&\quad + \left( \frac{\eta}{m^2} + 24 m L^2 \eta^3 \rho_t^2 \right) \sum_{i=1}^m \frac{\sigma_i^2}{B_i} \\
&\le \mathbb{E} \left[ f(\bar{x}^t) \right] + \frac{3 m L^2 \eta}{2} \, \mathbb{E} \left[ C(t) \right] 
- \frac{\eta}{2} \mathbb{E} \left[ \left\| \nabla f(\bar{x}^t) \right\|^2 \right]
+ \frac{2 \eta}{m^2} \sum_{i=1}^m \frac{\sigma_i^2}{B_i} \\
&= \mathbb{E} \left[ \Phi(t) \right]
- \frac{\eta}{2} \mathbb{E} \left[ \left\| \nabla f(\bar{x}^t) \right\|^2 \right]
+ \frac{2 \eta}{m^2} \sum_{i=1}^m \frac{\sigma_i^2}{B_i} - \frac{m L^2 \eta}{2} \, \mathbb{E} \left[ C(t) \right],
\end{align*}
where the last inequality follows from the choice of $R_t$. 
Specifically, by Lemma~\ref{lemma_1.1}, letting
\begin{equation*}
\begin{aligned}
\label{eq_b5.4}
R_t = \left\lceil \frac{2+\sqrt{2}}{2\sqrt{\chi}}\log\left( 14 \max\left\{ 9m^4, 70m^2, 6m^3\right\}\right)\right\rceil,
\end{aligned}
\end{equation*}
for $t \ge 1$, it holds that
\begin{align*}
& \frac{\eta}{2} - \frac{L \eta^2}{2} - 36 m^4 L^4 \eta^5 \rho_t^2 \ge 0, \\
& 4 m L^2 \eta \left( 27 m^2 L^2 \eta^2 + 2 \right) \rho_t^2 \le \frac{m L^2 \eta}{2}, \\
&24 m L^2 \eta^3 \rho_t^2 \le \frac{\eta}{m^2}.
\end{align*}

From the recursive relation of the Lyapunov function, we obtain the following estimate for the average gradient norm
\begin{align*}
\frac{1}{T} \sum_{t=0}^{T-1} \mathbb{E} \left[ \left\| \nabla f(\bar{x}^t) \right\|^2 \right] 
&\le \mathbb{E} \left[ \frac{2 \left( \Phi(0) - \Phi(T) \right)}{\eta T} \right] 
+ \frac{4}{m^2} \sum_{i=1}^m \frac{\sigma_i^2}{B_i} -\frac{1}{T} \sum_{t=0}^{T-1} mL^2 \mathbb{E} \left[ C(t) \right].
\end{align*}
The difference $\Phi(0)-\Phi(T)$ of the Lyapunov function in the above inequality can be bounded as
\begin{align*}
\mathbb{E} \left[ \Phi(0) - \Phi(T) \right] 
&\le f(\bar{x}^0) - f^* + 2 m L^2 \eta^3 \, \mathbb{E} \left[ \left\| S^0 - \mathbf{1} \bar{s}^0 \right\|^2 \right] \le \Delta + \frac{\epsilon^2 \eta T}{16},
\end{align*}
where the inequality holds by the parameter setting of 
\begin{align*}
R_0 = \left\lceil \frac{2+\sqrt{2}}{2\sqrt{\chi}}\log\left( 1 + \frac{448mL^2\eta^2 \sum_{i=1}^m \|y_i^0\|^2}{T \epsilon^2} \right) \right\rceil.
\end{align*}

Then the expectation of the averaged gradient norm satisfies
\begin{align*}
\frac{1}{T} \sum_{t=0}^{T-1} \mathbb{E} \left[ \left\| \nabla f(\bar{x}^t) \right\|^2 \right] 
&\le \mathbb{E} \left[ \frac{2 \left( \Phi(0) - \Phi(T) \right)}{\eta T} \right] 
+ \frac{4}{m^2} \sum_{i=1}^m \frac{\sigma_i^2}{B_i} -\frac{1}{T} \sum_{t=0}^{T-1} mL^2 \, \mathbb{E} \left[ C(t) \right] \\
&\le \frac{2 \Delta}{\eta T} + \frac{\epsilon^2}{8} +\frac{4}{m^2} \sum_{i=1}^m \frac{\sigma_i^2}{B_i} -\frac{1}{T} \sum_{t=0}^{T-1} mL^2 \mathbb{E} \left[ C(t) \right].\\
\end{align*}

Therefore, the final output of the algorithm satisfies the following error bound:
\begin{align*}
& \quad \mathbb{E} \left[ \left\|\nabla f\left(x_{i, \text{out}}\right)\right\|^2 \right] \\
&= \frac{1}{T} \sum_{t=0}^{T-1} \mathbb{E}\left\|\nabla f\left(x_i^t\right)\right\|^2 \\
&\leq \frac{2}{T} \sum_{t=0}^{T-1} \mathbb{E}\left[\| \nabla f(\bar{x}^t) \|^2 + \| \nabla f(x_i^t) - \nabla f(\bar{x}^t) \|^2\right] \\
&\leq \frac{2}{T} \sum_{t=0}^{T-1} \mathbb{E}\left[\|\nabla f(\bar{x}^t)\|^2 + L^2 \left\|x_i^t - \bar{x}^t\right\|^2\right] \\
&\leq \frac{2}{T} \sum_{t=0}^{T-1} \mathbb{E} \|\nabla f(\bar{x}^t)\|^2 + \frac{2 L^2}{T} \sum_{t=0}^{T-1} \mathbb{E} \|X^t - \mathbf{1} \bar{x}^t\|^2 \\
&\leq \frac{4 \Delta}{\eta T} + \frac{\epsilon^2}{4} +\frac{8}{m^2} \sum_{i=1}^m \frac{\sigma_i^2}{B_i} \\
&\leq \frac{\epsilon^2}{4} + \frac{\epsilon^2}{4} + \frac{\epsilon^2}{2} \\
& = \epsilon^2,
\end{align*}
where the last inequality is based on the parameter settings in Theorem~\ref{thm_2.1}.

The sampling complexity of the algorithm is bounded by
\begin{align*}
\left( \sum_{i=1}^m B_i \right) T 
&= \left( \sum_{i=1}^m \left\lceil \frac{16\sigma_i \sum_{j=1}^{m}\sigma_j}{m^2 \epsilon^2} \right\rceil \right) \cdot \left\lceil \frac{32 \Delta L}{\epsilon^2} \right\rceil \\
&= O \left( \frac{\Delta L\, \bar{\sigma}_{\mathrm{AM}}^2}{\epsilon^4} + \frac{m \Delta L}{\epsilon^2} \right),
\end{align*}
and the communication complexity is bounded by
\begin{align*}
\sum_{i=0}^{T-1} R_i
&= \tilde{O} \left( \frac{\Delta L}{\sqrt{\chi}\epsilon^2}\right).
\end{align*}

\section{The Proof of Theorem~\ref{thm_3.1}}
In this section, we establish the complexity lower bound stated in Theorem~\ref{thm_3.1}. 
\citet{arjevani2023lower} establish the sampling complexity lower bound for stochastic optimization algorithms in the single‑machine setting, with the proof based on the construction of the following function:
\begin{align}
\label{eq_FD}
    F_D(x) &:= -\psi(1) \phi([x]_1) + \sum_{i=2}^D \left[ \psi(-[x]_{i-1}) \phi(-[x]_i) - \psi([x]_{i-1}) \phi([x]_i) \right],
\end{align}
where $D$ is the dimension of $x$, $[x]_i$ denotes the $i$‑th coordinate of $x$, and the functions $\psi(\cdot)$, $\phi(\cdot)$ are defined as
\begin{align*}
   \psi(t) = \begin{cases}
        0, & t \leq \frac{1}{2} \\
        \exp\left(1 - \frac{1}{(2t-1)^2}\right), & t > \frac{1}{2}
    \end{cases}, \qquad \phi(t) = \int_{-\infty}^t e^{-\frac{1}{2}\tau^2} d\tau. 
\end{align*}

This function has the following properties.
\begin{lemma}[{\citet{carmon2020lower}}]
\label{lemma_2.1}
The function $F_D$ satisfies the following properties:
\begin{enumerate}
\item $F_D(0) - \inf_x F_D(x) \leq \Delta_0 D$, where $\Delta_0 = 12$;
\item $F_D$ is $\ell_1$‑Lipschitz continuous with $\ell_1 = 152$;
\item If $[x]_D = 0$, then $\|\nabla F_D(x)\| > 1$.
\end{enumerate}
\end{lemma}

The stochastic gradient in the single‑machine setting \cite{arjevani2023lower} is constructed as
\begin{align}
\label{eq_G1}
\left[ G_D(x, \xi; p) \right]_i := \nabla_i F_D(x) \cdot 
\left( 
1 + \mathbb{I}\left\{ i > \mathrm{prog}_{0}(x) \right\} \cdot \left( \frac{\xi}{p} - 1 \right)
\right),
\end{align}
where $\xi \sim \mathrm{Bernoulli}(p)$ and $\mathrm{prog}_{0}(x) := \max\{i \ge 0| |[x]_i| > 0\}$.
The stochastic gradient $G_D$ is unbiased and has bounded variance.
\begin{lemma}[{\citet{arjevani2023lower}}]
\label{lemma_2.2}
The stochastic gradient $G_D(x,\xi;p)$ satisfies the following properties:
\begin{enumerate}
\item $\mathbb{E}_{\xi}\left[G_D(x,\xi;p)\right] = \nabla F_D(x)$;
\item The variance satisfies
\begin{align*}
\mathbb{E}_{\xi} \left[\left\| G_D(x,\xi;p) - \nabla F_D(x) \right\|^2\right]
\leq a^2 \cdot \frac{1-p}{p},
\end{align*}
for all $x \in \mathbb{R}^D$, where $a = 23$.
\end{enumerate}
\end{lemma}
Moreover, the stochastic gradient $G_D(x,\xi;p)$ forms a probability‑$p$ zero‑chain. Informally, this means that starting from $x^{(0)} = 0$, the number of nonzero coordinates can increase by at most one at each iteration with probability $p$. This yields the following key lemma.
\begin{lemma}[{\citet[Lemma 1]{arjevani2023lower}}]
\label{lemma_2.3}
If the total number of samples used during the first $t$ iterations does not exceed $(D-1)/(2p)$, then the probability that $[x^t]_D = 0$ is at least $1/2$.
\end{lemma}

Building on the above function, we extend the construction to the distributed setting. Existing lower bounds for distributed optimization \cite{lu2021optimal,lu2023decentralized,yuan2022revisiting} commonly set $f_1 = f_2 = \dots = f_m = F_D$, which is insufficient to capture the difficulties introduced by heterogeneous variance. To overcome this limitation, we construct orthogonal local functions satisfying $\langle \nabla f_i(x), \nabla f_j(x) \rangle = 0$ for all $i \neq j$. Furthermore, to address the asymmetry caused by heterogeneous variance, we explicitly account for the sampling differences across nodes.

\begin{proof}[Proof of Theorem~\ref{thm_3.1}]
Let $\#\text{SFO}$ denote the total number of stochastic first-order oracle calls of the algorithm, and let $\#\text{SFO}_i$ denote the number of oracle calls performed at node $i$. For any algorithm $\mathcal{A}$, we have
\begin{align*}
\mathbb{E} \left[ \#\text{SFO}_i \right] = c_i \cdot \#\text{SFO}, \quad \text{where} \quad c_i > 0 \quad \text{and} \quad \sum_{i=1}^m c_i = 1.
\end{align*}
Letting $D = \sum_i D_i$, we define a matrix sequence $\{U_i\}_{i=1}^m$ such that $U_i \in \mathbb{R}^{D_i\times D}$, $U_i U_i^\top = I$, and $U_i U_j^\top = \mathbf{0}$ for any $1\le i \ne j \le m$.
The functions are constructed as
\begin{align*}
f_i(x) &= \frac{m L \lambda_i^2}{\ell} \, F_{D_i} \left( \frac{U_i x}{\lambda_i} \right), \quad
g_i(x, \xi; p_i) = \frac{m L \lambda_i}{\ell} \, U_i^\top G_{D_i} \left( \frac{U_i x}{\lambda_i}, \xi; p_i \right), \\
\text{where} & \quad \ell = \ell_1, \quad \lambda_i = \frac{\ell}{\sqrt{m}L} \sqrt{\frac{\sigma_i}{\bar{\sigma}_{\mathrm{AM}}}} \cdot 2 \epsilon, \quad
\frac{1}{p_i} = \frac{\sigma_i \bar{\sigma}_{\mathrm{AM}}}{4 m a^2 \epsilon^2} + 1.
\end{align*}

First, we verify that the functions constructed above satisfy Assumption~\ref{ass_2.3} on smoothness. For the global function $f$, we have
\begin{align*}
\left\| \nabla f(x) - \nabla f(y) \right\|^2 
&= \left\| \sum_{i=1}^m \frac{L \lambda_i}{\ell} \, U_i^\top \left( \nabla F_{D_i} \left( \frac{U_i x}{\lambda_i} \right) - \nabla F_{D_i} \left( \frac{U_i y}{\lambda_i} \right) \right) \right\|^2 \\
&= \sum_{i=1}^m \frac{L^2 \lambda_i^2}{\ell^2} \left\| \nabla F_{D_i} \left( \frac{U_i x}{\lambda_i} \right) - \nabla F_{D_i} \left( \frac{U_i y}{\lambda_i} \right) \right\|^2 \\
&\le L^2 \sum_{i=1}^m \left\| U_i (x - y) \right\|^2 \\
&\le L^2 \left\| x - y \right\|^2.
\end{align*}
For the local function $f_i$, it holds that
\begin{align*}
\left\| \nabla f_i(x) - \nabla f_i(y) \right\|^2 
&= \left\| \frac{m L \lambda_i}{\ell} \, U_i^\top \left( \nabla F_{D_i} \left( \frac{U_i x}{\lambda_i} \right) - \nabla F_{D_i} \left( \frac{U_i y}{\lambda_i} \right) \right) \right\|^2 \\
&= \frac{m^2 L^2 \lambda_i^2}{\ell^2} \left\| \nabla F_{D_i} \left( \frac{U_i x}{\lambda_i} \right) - \nabla F_{D_i} \left( \frac{U_i y}{\lambda_i} \right) \right\|^2 \\
&\le m^2 L^2 \left\| U_i (x - y) \right\|^2 \\
&\le m^2 L^2 \left\| x - y \right\|^2.
\end{align*}

Next, we verify that the stochastic gradients satisfy Assumption~\ref{ass_2.2}. By Property 1 of Lemma~\ref{lemma_2.2}, the unbiased property of $g_i(x,\xi; p_i)$ follows directly. Moreover, the variance has the following upper bound:
\begin{align*}
\mathbb{E} \left[ \left\| g_i(x, \xi) - \nabla f_i(x) \right\|^2 \right]
&= \mathbb{E} \left[ \left\| \frac{m L \lambda_i}{\ell} \, U_i^\top 
\left( G_{D_i} \left( \frac{U_i x}{\lambda_i}, \xi; p_i \right)
- \nabla F_{D_i} \left( \frac{U_i x}{\lambda_i} \right) \right) \right\|^2 \right] \\
&= \frac{m^2 L^2 \lambda_i^2}{\ell^2} 
\mathbb{E} \left[ \left\| G_{D_i} \left( \frac{U_i x}{\lambda_i}, \xi; p_i \right) 
- \nabla F_{D_i} \left( \frac{U_i x}{\lambda_i} \right) \right\|^2 \right] \\
&\le 4 m \epsilon^2 \cdot \frac{\sigma_i}{\bar{\sigma}_{\mathrm{AM}}} \cdot 
\frac{a^2 (1 - p_i)}{p_i} \\
&= \sigma_i^2,
\end{align*}
where the inequality follows from Property 2 of Lemma~\ref{lemma_2.2}.

Letting $D_i = \left\lfloor \dfrac{\Delta L}{4 \Delta_0 \ell \epsilon^2} \cdot \dfrac{\bar{\sigma}_{\mathrm{AM}}}{\sigma_i} \cdot m c_i \right\rfloor$, we have
\begin{align*}
f(0) - \inf_x f(x) 
&= \frac{1}{m} \sum_{i=1}^m f_i(0) - \inf_x \frac{1}{m} \sum_{i=1}^m f_i(x) \\
&\le \frac{1}{m} \sum_{i=1}^m f_i(0) - \frac{1}{m} \sum_{i=1}^m \inf_x f_i(x) \\
&= \sum_{i=1}^m \frac{4 \epsilon^2 \ell}{m L} \cdot \frac{\sigma_i}{\bar{\sigma}_{\mathrm{AM}}} 
\left( F_{D_i}(0) - \inf_x F_{D_i}(x) \right) \\
&\le \sum_{i=1}^m \frac{4 \epsilon^2 \ell}{m L} \cdot \frac{\sigma_i}{\bar{\sigma}_{\mathrm{AM}}} \cdot \Delta_0 D_i \\
&\le \sum_{i=1}^m \frac{4 \epsilon^2 \ell}{m L} \cdot \frac{\sigma_i}{\bar{\sigma}_{\mathrm{AM}}} \cdot 
\frac{\Delta L}{4 \Delta_0 \ell \epsilon^2} \cdot \frac{\bar{\sigma}_{\mathrm{AM}}}{\sigma_i} \cdot m c_i \\
&= \Delta \sum_{i=1}^m c_i \\
&= \Delta.
\end{align*}
Therefore, under this choice of $D_i$, Assumption~\ref{ass_2.1} is satisfied. 
This completes the verification that the constructed functions satisfy Assumptions 1-3.

If $\#\text{SFO} \le \dfrac{1}{128 \Delta_0 \ell a^2}\cdot \dfrac{\Delta L \bar{\sigma}_{\mathrm{AM}}^2}{\epsilon^4}$, and $\epsilon \le \min_i \sqrt{\dfrac{\Delta L \bar{\sigma}_{\mathrm{AM}}}{8 \Delta_0 \ell \sigma_i} m c_i}$, then we have
\begin{align*}
\mathbb{E} \left[ \left\| \nabla f(x) \right\|^2 \right]
&= \mathbb{E} \left[ \left\| \frac{1}{m} \sum_{i=1}^m \nabla f_i(x) \right\|^2 \right] \\
&= \mathbb{E} \left[ \left\| \sum_{i=1}^m \frac{L \lambda_i}{\ell} 
U_i^\top \nabla F_{D_i} \left( \frac{U_i x}{\lambda_i} \right) \right\|^2 \right] \\
&= \frac{4 \epsilon^2}{m} \sum_{i=1}^m \frac{\sigma_i}{\bar{\sigma}_{\mathrm{AM}}} 
\mathbb{E} \left[ \left\| \nabla F_{D_i} \left( \frac{U_i x}{\lambda_i} \right) \right\|^2 \right] \\
&> \frac{4 \epsilon^2}{m} \sum_{i=1}^m \frac{\sigma_i}{\bar{\sigma}_{\mathrm{AM}}} 
\mathbb{P} \left( \left[ U_i x \right]_{D_i} = 0 \right) \\
&= \frac{4 \epsilon^2}{m} \sum_{i=1}^m \frac{\sigma_i}{\bar{\sigma}_{\mathrm{AM}}} 
\mathbb{P} \left( \left[ U_i x \right]_{D_i} > 0 \,\middle|\, 
\#\text{SFO}_i < \frac{D_i - 1}{2 p_i} \right) 
\mathbb{P} \left( \#\text{SFO}_i < \frac{D_i - 1}{2 p_i} \right) \\
&\ge \frac{4 \epsilon^2}{m} \sum_{i=1}^m \frac{\sigma_i}{\bar{\sigma}_{\mathrm{AM}}} \cdot \frac{1}{2} \cdot \frac{1}{2} \\
&= \epsilon^2.
\end{align*}
The final inequality above follows from Lemma~\ref{lemma_2.3} and the following probability lower bound:
\begin{align*}
\mathbb{P} \left( \#\text{SFO}_i < \frac{D_i - 1}{2 p_i} \right) 
&= 1 - \mathbb{P} \left( \#\text{SFO}_i \ge \frac{D_i - 1}{2 p_i} \right) \\
&\ge 1 - \frac{\mathbb{E} \left[ \#\text{SFO}_i \right]}{\dfrac{D_i - 1}{2 p_i}} \\
&\ge 1 - \frac{\dfrac{1}{128 \Delta_0 \ell a^2} \cdot \dfrac{\Delta L \bar{\sigma}_{\mathrm{AM}}^2}{\epsilon^4} \cdot c_i}
{\dfrac{\Delta L}{8 \Delta_0 \ell \epsilon^2} \cdot \dfrac{\bar{\sigma}_{\mathrm{AM}}}{\sigma_i} \cdot m c_i \cdot \dfrac{\sigma_i \bar{\sigma}_{\mathrm{AM}}}{8 m a^2 \epsilon^2}} \\
&= \frac{1}{2},
\end{align*}
where the first inequality follows from Markov’s inequality.
Therefore, to ensure $\mathbb{E}\left[\|\nabla f(x)\|^{2}\right] \le \epsilon^{2}$, the number of samples must be at least $\Omega(\Delta L \bar{\sigma}_{\mathrm{AM}}^2 \epsilon^{-4})$.

For the second term $\Omega(m \Delta L \epsilon^{-2})$ in Theorem~\ref{thm_3.1}, which is independent of the stochastic gradients, we construct the following function:
\begin{align*}
f_i(x) &= \frac{m L \lambda_i^2}{\ell} \, F_{D_i} \left( \frac{U_i x}{\lambda_i} \right), \quad
g_i(x, \xi; p_i) = \nabla f_i(x),
\end{align*}
where $\ell = \ell_1$, $\lambda_i = 2\epsilon \ell/(\sqrt{m}L)$. 
It is straightforward to verify that it satisfies Assumptions 2 and 3.

Letting $D_i = D = \left\lfloor \dfrac{\Delta L}{4 \Delta_0 \ell \epsilon^2} \right\rfloor$, it holds that $f(0) - \inf_x f(x) \le \Delta$. If $\#\text{SFO} \le mD/16$, at least half of the nodes satisfy $[U_i x]_{D} = 0$, then we have
\begin{align*}
\mathbb{E} \left[ \left\| \nabla f(x) \right\|^2 \right]
&= \mathbb{E} \left[ \left\| \frac{1}{m} \sum_{i=1}^m \nabla f_i(x) \right\|^2 \right] \\
&= \mathbb{E} \left[ \left\| \sum_{i=1}^m \frac{L \lambda_i}{\ell} 
U_i^\top \nabla F_{D_i} \left( \frac{U_i x}{\lambda_i} \right) \right\|^2 \right] \\
&= \frac{4 \epsilon^2}{m} \sum_{i=1}^m 
\mathbb{E} \left[ \left\| \nabla F_{D_i} \left( \frac{U_i x}{\lambda_i} \right) \right\|^2 \right] \\
&> \frac{4 \epsilon^2}{m} \cdot \frac{m}{2} \\
&> \epsilon^2.
\end{align*}
Therefore, the sample complexity is lower bounded by $\Omega(m \Delta L \epsilon^{-2})$.

Combining the two terms completes the proof of Theorem~\ref{thm_3.1}.
\end{proof}

\section{The Proof of Theorem~\ref{thm_4.1}}
For the proof of Theorem~\ref{thm_4.1}, we first introduce the Lyapunov function
\begin{equation*}
\Phi(t)=f(\bar{x}^t)+\frac{\eta}{p} U(t)+\frac{\eta}{m p} V(t)+\frac{1}{\eta} C(t),
\end{equation*}
where the terms are defined as 
\begin{itemize}
    \item Global gradient estimation error: $U(t) = \left\| \frac{1}{m} \sum_{i=1}^m \left( y_i^t - \nabla f_i(x_i^t) \right) \right\|^2$; 
    \item average local gradient estimation error: $V(t) = \frac{1}{m} \left\| Y^t - \nabla \mathbf{f}(X^t) \right\|^2$; 
    \item consensus error: $C(t) = \left\| X^t - \mathbf{1} \bar{x}^t \right\|^2 + \eta^2 \left\| S^t - \mathbf{1} \bar{s}^t \right\|^2$. 
\end{itemize}
The overall proof strategy involves leveraging the smoothness of the function to derive the descent lemma, then bounding $U(t)$, $V(t)$, and $C(t)$ respectively, ultimately yielding an upper bound on the gradient
\begin{equation*}
\frac{\eta}{2} \mathbb{E} \left[\|\nabla f(\bar{x}^t)\|^2 \right] \le \mathbb{E}\left[\Phi(t)-\Phi(t+1)-\frac{1}{2 \eta} C(t)+\sum_{i=1}^m \frac{3 \eta \sigma_i^2}{m^2 B_i}\right].
\end{equation*}
Averaging the above inequality over $t = 0, \dots, T-1$ yields
\begin{equation*}
\frac{\eta}{2 T} \sum_{t=0}^{T-1} \mathbb{E}\left\|\nabla f\left(\bar{x}^t\right)\right\|^2 
\leq \mathbb{E}\left[\frac{\Phi(0)-\Phi(T)}{T}-\frac{1}{2 \eta T} \sum_{t=0}^{T-1} C(t)\right] 
+ \frac{3 \eta \sigma_i^2}{m^2 B_i}.
\end{equation*}
Under the parameter choices specified in Theorem~\ref{thm_4.1}, the resulting upper bound on $\sum_{t=0}^{T-1} \mathbb{E}\left\|\nabla f(\bar{x}^t)\right\|^2$ ensures that the output satisfies
$\mathbb{E}\left[\left\|\nabla f\left(x_{i,\mathrm{out}}\right)\right\|^2\right] \leq \epsilon^2$. 

We first present several key lemmas.

\begin{lemma}[Descent Lemma]
\label{lemma_descent}
For Algorithm~\ref{algo:d-nss-vr}, the following descent relation holds:
\begin{equation*}
f(\bar{x}^{t+1})
\leq f(\bar{x}^t)
- \frac{\eta}{2} \| \nabla f(\bar{x}^t)\|^2
- \left(\frac{1}{2\eta} - \frac{\bar{L}}{2} \right) \| \bar{x}^{t+1} - \bar{x}^t \|^2
+ \eta U(t)
+ \bar{L}^2 \eta C(t).
\end{equation*}
\end{lemma}

\begin{proof}
By the $\bar{L}$‑smoothness of $f$, we have
\begin{equation}
\label{eq_b1_1}
f(\bar{x}^{t+1}) \leq f(\bar{x}^t) + \left\langle \nabla f(\bar{x}^t), \bar{x}^{t+1} - \bar{x}^t \right\rangle + \frac{\bar{L}}{2} \|\bar{x}^{t+1} - \bar{x}^t\|^2.
\end{equation}
Noting that $\bar{x}^{t+1} = \bar{x}^t - \eta \bar{s}^t$, substituting this into \eqref{eq_b1_1} yields
\begin{equation}
\begin{aligned}
f(\bar{x}^{t+1}) 
\leq\; & f(\bar{x}^t) - \eta \left\langle \nabla f(\bar{x}^t), \bar{s}^t \right\rangle + \frac{\bar{L} \eta^2}{2} \|\bar{s}^t\|^2 \\
=\; & f(\bar{x}^t) - \frac{\eta}{2} \| \nabla f(\bar{x}^t) \|^2 - \left( \frac{\eta}{2} - \frac{\bar{L} \eta^2}{2} \right) \|\bar{s}^t\|^2 + \frac{\eta}{2} \| \bar{s}^t - \nabla f(\bar{x}^t) \|^2 \\
=\; & f(\bar{x}^t) - \frac{\eta}{2} \|\nabla f(\bar{x}^t)\|^2 - \left( \frac{1}{2\eta} - \frac{\bar{L}}{2} \right) \|\bar{x}^{t+1} - \bar{x}^t\|^2 + \frac{\eta}{2} \| \bar{s}^t - \nabla f(\bar{x}^t) \|^2.
\end{aligned}
\end{equation}
We next bound the last term, which gives
\begin{equation}
\label{eq_b1_2}
\begin{aligned}
\left\| \bar{s}^t - \nabla f(\bar{x}^t) \right\|^2
= \; & \left\| \frac{1}{m} \sum_{i=1}^m y_i^t - \nabla f_i(\bar{x}^t) \right\|^2 \\
\le \; & 2 \left\| \frac{1}{m} \sum_{i=1}^m \left( y_i^t - \nabla f_i(x_i^t) \right) \right\|^2 
+ 2 \left\| \frac{1}{m} \sum_{i=1}^m \left( \nabla f_i(x_i^t) - \nabla f_i(\bar{x}^t) \right) \right\|^2 \\
\le \; & 2 \left\| \frac{1}{m} \sum_{i=1}^m \left( y_i^t - \nabla f_i(x_i^t) \right) \right\|^2 
+ \frac{2}{m} \sum_{i=1}^m \left\| \nabla f_i(x_i^t) - \nabla f_i(\bar{x}^t) \right\|^2 \\
\le \; & 2 \left\| \frac{1}{m} \sum_{i=1}^m \left( y_i^t - \nabla f_i(x_i^t) \right) \right\|^2 
+ 2 \bar{L}^2 \left\| X^t - \mathbf{1} \bar{x}^t \right\|^2 \\
\le \; & 2 U(t) + 2 \bar{L}^2 C(t),
\end{aligned}
\end{equation}
where the third inequality follows from the $\sqrt{m}\bar{L}$‑smoothness of each $f_i$, and the last inequality follows directly from the definitions of $U(t)$ and $C(t)$.

Combining \eqref{eq_b1_1} and \eqref{eq_b1_2} completes the proof of Lemma~\ref{lemma_descent}.
\end{proof}

\begin{lemma}
\label{lemma_C}
Under the parameter choices specified in Theorem~\ref{thm_4.1}, for Algorithm~\ref{algo:d-nss-vr}, we have
\begin{equation*}
\mathbb{E}[C(t+1)]
\leq 38 c m \rho_t^2 \, \mathbb{E}[C(t)]
+ 6 \rho_t^2 \eta^2 m \, \mathbb{E}[V(t)]
+ 24 c \rho_t^2 m^2 \, \mathbb{E}\|\bar{x}^{t+1} - \bar{x}^t\|^2
+ \sum_{i=1}^m \frac{6 p \rho_t^2 \eta^2 \sigma_i^2}{B_i},
\end{equation*}
where $c = \max\{1/(bq),\, 1\}$.
\end{lemma}

\begin{proof}
For the consensus error, we have
\begin{equation*}
\begin{aligned}
\left\| X^{t+1} - \mathbf{1} \bar{x}^{t+1} \right\| 
&= \left\| \mathtt{FastMix}(X^t - \eta S^t, W, R_t) - \frac{\mathbf{1} \mathbf{1}^\top}{m} \mathtt{FastMix}(X^t - \eta S^t, W, R_t) \right\| \\
&\leq \rho_t \left\| X^t - \eta S^t - \mathbf{1}(\bar{x}^t - \eta \bar{s}^t) \right\| \\
&\leq \rho_t \left\| X^t - \mathbf{1} \bar{x}^t \right\| + \rho_t \eta \left\| S^t - \mathbf{1} \bar{s}^t \right\|,
\end{aligned}
\end{equation*}
where the first inequality follows from Lemma \ref{lemma_1.1}, and the second inequality follows from the triangle inequality. By Young's inequality, we have
\begin{equation}
\label{eq_b1_5}
\begin{aligned}
\|X^{t+1} - \mathbf{1} \bar{x}^{t+1}\|^2 
&\leq 2 \rho_t^2 \|X^t - \mathbf{1} \bar{x}^t\|^2 + 2 \rho_t^2 \eta^2 \|S^t - \mathbf{1} \bar{s}^t\|^2 \\
&= 2 \rho_t^2 C(t).
\end{aligned}
\end{equation}

Similarly, from the update rule of the algorithm, we have
\begin{equation}
\label{eq_b1_6}
\begin{aligned}
\|S^{t+1} - \mathbf{1} \bar{s}^{t+1}\|^2 
&= \left\| \mathtt{FastMix}(S^t + Y^{t+1} - Y^t, W, R_t) - \frac{\mathbf{1} \mathbf{1}^\top}{m} \mathtt{FastMix}(S^t + Y^{t+1} - Y^t, W, R_t) \right\|^2 \\
&\leq \rho_t^2 \left\| S^t + Y^{t+1} - Y^t - \mathbf{1}(\bar{s}^t + \bar{y}^{t+1} - \bar{y}^t) \right\|^2 \\
&\leq 2 \rho_t^2 \|S^t - \mathbf{1} \bar{s}^t\|^2 + 2 \rho_t^2 \|Y^{t+1} - Y^t - \mathbf{1}(\bar{y}^{t+1} - \bar{y}^t)\|^2 \\
&\leq 2 \rho_t^2 \|S^t - \mathbf{1} \bar{s}^t\|^2 + 2 \rho_t^2 \|Y^{t+1} - Y^t\|^2,
\end{aligned}
\end{equation}
where the second inequality follows from the Young's inequality, and the last inequality uses the fact that
$\|A - \frac{1}{m} \mathbf{1} \mathbf{1}^\top A\| \leq \|A\|$ for any $\mathbf{z} \in \mathbb{R}^{m \times d}$.

We next bound the expectation of the term $\rho_t^2 \|Y^{t+1} - Y^t\|^2$. Noting that
\begin{equation*}
\mathbb{E} \left[ \rho_t^2 \|Y^{t+1} - Y^t\|^2 \right] = \sum_{i=1}^m \mathbb{E} \left[ \rho_t^2 \| y_i^{t+1} - y_i^t \|^2 \right],
\end{equation*}
and for each $i$, it holds that
\begin{equation*}
\begin{aligned}
\mathbb{E} \left[ \rho_t^2 \| y_i^{t+1} - y_i^t \|^2 \right] 
\leq\; & p \rho_t^2 \, \mathbb{E} \left\| \frac{1}{B_i} \sum_{\xi_{i,j} \in \mathcal{S}_i(t)} g_i(x_i^{t+1}; \xi_{i,j}) - y_i^t \right\|^2 \\
& + (1 - p) \frac{\rho_t^2}{b q} \, \mathbb{E} \left\| g_i(x_i^{t+1}; \xi_{i,j}) - g_i(x_i^t; \xi_{i,j}) \right\|^2.
\end{aligned}
\end{equation*}
By Young's inequality, we have
\begin{equation*}
\begin{aligned}
& \mathbb{E} \left[ \rho_t^2 \| y_i^{t+1} - y_i^t \|^2 \right] \\
\leq\; & 3 p \rho_t^2 \, \mathbb{E} \left\| \frac{1}{B_i} \sum_{\xi_{i,j} \in \mathcal{S}_i(t)} g_i(x_i^{t+1}; \xi_{i,j}) - \nabla f_i(x_i^{t+1}) \right\|^2 \\
& + 3 p \rho_t^2 \, \mathbb{E} \left\| \nabla f_i(x_i^{t+1}) - \nabla f_i(x_i^t) \right\|^2 \\
& + 3 p \rho_t^2 \, \mathbb{E} \left\| \nabla f_i(x_i^t) - y_i^t \right\|^2 
+ \frac{(1 - p) \rho_t^2}{b q} \mathbb{E} \left\| g_i(x_i^{t+1}; \xi_{i,j}) - g_i(x_i^t; \xi_{i,j}) \right\|^2.
\end{aligned}
\end{equation*}
Combining with Assumption~\ref{ass_2.2} and the mean‑square smoothness, we obtain
\begin{align*}
& \mathbb{E} \left[ \rho_t^2 \| Y^{t+1} - Y^t \|^2 \right] \notag\\
\leq\; & \sum_{i=1}^m \frac{3 p \rho_t^2 \sigma_i^2}{B_i} + 3 p \rho_t^2 \, \mathbb{E} \| \nabla \mathbf{f}(X^t) - Y^t \|^2 \notag\\
& + \left(3 p \rho_t^2 + \frac{(1 - p) \rho_t^2}{b q} \right) m\bar{L}^2 \, \mathbb{E} \| X^{t+1} - X^t \|^2 \notag\\
\leq\; & \sum_{i=1}^m \frac{3 p \rho_t^2 \sigma_i^2}{B_i} + 3 p \rho_t^2 \, \mathbb{E} \| \nabla \mathbf{f}(X^t) - Y^t \|^2 \notag\\
& + 3\left(3 p \rho_t^2 + \frac{(1 - p) \rho_t^2}{b q} \right) m\bar{L}^2 \left( 2 \mathbb{E} \| X^t - \mathbf{1} \bar{x}^t \|^2 + 2 \eta^2 \mathbb{E} \| S^t - \mathbf{1} \bar{s}^t \|^2 \right) \notag\\
& + 3\left(3 p \rho_t^2 + \frac{(1 - p) \rho_t^2}{b q} \right) m\bar{L}^2 \left[ m \, \mathbb{E} \| \bar{x}^{t+1} - \bar{x}^t \|^2 + \mathbb{E} \| X^t - \mathbf{1} \bar{x}^t \|^2 \right] \notag\\
=\; & \sum_{i=1}^m \frac{3 p \rho_t^2 \sigma_i^2}{B_i} + 3 p \rho_t^2 \, \mathbb{E} \| \nabla \mathbf{f}(X^t) - Y^t \|^2 \notag\\
& + 9 \left(3 p \rho_t^2 + \frac{(1 - p) \rho_t^2}{b q} \right) m\bar{L}^2 \mathbb{E} \| X^t - \mathbf{1} \bar{x}^t \|^2 \notag\\
& + 6 \left(3 p \rho_t^2 + \frac{(1 - p) \rho_t^2}{b q} \right) m\bar{L}^2 \eta^2 \mathbb{E} \| S^t - \mathbf{1} \bar{s}^t \|^2 \notag\\
& + 3 \left(3 p \rho_t^2 + \frac{(1 - p) \rho_t^2}{b q} \right) \bar{L}^2 m^2 \, \mathbb{E} \| \bar{x}^{t+1} - \bar{x}^t \|^2.
\end{align*}
Letting $c = \max\left\{ \frac{1}{b q}, 1 \right\}$, then we have
\begin{align}
\label{eq_b1_7}
\mathbb{E} \left[ \rho_t^2 \| Y^{t+1} - Y^t \|^2 \right]
\leq\; & \sum_{i=1}^m \frac{3 p \rho_t^2 \sigma_i^2}{B_i}
+ 3 p \rho_t^2 \, \mathbb{E} \left\| \nabla \mathbf{f}(X^t) - Y^t \right\|^2 \notag\\
& + 36 c \rho_t^2 m\bar{L}^2 \, \mathbb{E} \left\| X^t - \mathbf{1} \bar{x}^t \right\|^2 \notag\\
& + 24 c \rho_t^2 m\bar{L}^2 \eta^2 \, \mathbb{E} \left\| S^t - \mathbf{1} \bar{s}^t \right\|^2 \notag\\
& + 12 c \rho_t^2 \bar{L}^2 m^2 \, \mathbb{E} \left\| \bar{x}^{t+1} - \bar{x}^t \right\|^2.
\end{align}

Combining inequalities (\ref{eq_b1_5}), (\ref{eq_b1_6}), and (\ref{eq_b1_7}) yields
\begin{align*}
& \mathbb{E} \left[ \|X^{t+1} - \mathbf{1} \bar{x}^{t+1}\|^2 + \eta^2 \|S^{t+1} - \mathbf{1} \bar{s}^{t+1}\|^2 \right] \\
\leq\; & 2 \rho_t^2 \, \mathbb{E} \|X^t - \mathbf{1} \bar{x}^t\|^2 
+ 4 \rho_t^2 \eta^2 \, \mathbb{E} \|S^t - \mathbf{1} \bar{s}^t\|^2 
+ 2 \eta^2 \, \mathbb{E} \left[ \rho_t^2 \|Y^{t+1} - Y^t\|^2 \right] \\
\leq\; & 38 c m \rho_t^2 \, \mathbb{E} \left[ \|X^t - \mathbf{1} \bar{x}^t\|^2 + \eta^2 \|S^t - \mathbf{1} \bar{s}^t\|^2 \right] 
+ 6 \rho_t^2 \eta^2 \, \mathbb{E} \| \nabla \mathbf{f}(X^t) - Y^t \|^2 \\
& + 24 c \rho_t^2 m^2 \, \mathbb{E} \| \bar{x}^{t+1} - \bar{x}^t \|^2 
+ \sum_{i=1}^m \frac{6 p \rho_t^2 \eta^2 \sigma_i^2}{B_i},
\end{align*}
where the last inequality uses the stepsize condition $\eta \le 1/(2\bar{L})$. Therefore,
\begin{equation*}
\mathbb{E}[C(t+1)] \leq 38 c m \rho_t^2 \, \mathbb{E}[C(t)] 
+ 6 \rho_t^2 \eta^2 m \, \mathbb{E}[V(t)] 
+ 24 c \rho_t^2 m^2 \, \mathbb{E} \| \bar{x}^{t+1} - \bar{x}^t \|^2 
+ \sum_{i=1}^m \frac{6 p \rho_t^2 \eta^2 \sigma_i^2}{B_i},
\end{equation*}
which completes the proof of the lemma.
\end{proof}

\begin{lemma}
\label{lemma_V}
Under the parameter choices specified in Theorem~\ref{thm_4.1}, for Algorithm~\ref{algo:d-nss-vr}, we have
\begin{equation*}
\mathbb{E}[V(t+1)] 
\leq (1 - p) \, \mathbb{E}[V(t)] 
+ \frac{9(1 - p)\bar{L}^2}{b q} \, \mathbb{E}[C(t)] 
+ \frac{3(1 - p)\bar{L}^2}{b q} \, \mathbb{E}\|\bar{x}^{t+1} - \bar{x}^t\|^2 
+ \sum_{i=1}^m \frac{p \sigma_i^2}{mB_i}.
\end{equation*}
\end{lemma}

\begin{proof}
From the update rule of $y_i(t)$, we have
\begin{align}
\label{eq_b1_8}
& \mathbb{E} \left\| y_i^{t+1} - \nabla f_i(x_i^{t+1}) \right\|^2 \notag\\
=\; & p \, \mathbb{E} \left\| \frac{1}{B_i} \sum_{\xi_{i,j} \in \mathcal{S}_i^{\prime}(t)} g_i(x_i^{t+1}; \xi_{i,j}) - \nabla f_i(x_i^{t+1}) \right\|^2 \notag\\
& + (1 - p) \, \mathbb{E} \left\| y_i^t + \frac{\omega_i^t}{b q} \sum_{\xi_{i,j} \in \mathcal{S}_i(t)} \left( g_i(x_i^{t+1}; \xi_{i,j}) - g_i(x_i^t; \xi_{i,j}) \right) - \nabla f_i(x_i^{t+1}) \right\|^2 \notag\\
\leq\; & \frac{p \sigma_i^2}{B_i} + (1 - p) \, \mathbb{E} \left\| y_i^t + \frac{\omega_i^t}{b q} \sum_{\xi_{i,j} \in \mathcal{S}_i(t)} \left( g_i(x_i^{t+1}; \xi_{i,j}) - g_i(x_i^t; \xi_{i,j}) \right) - \nabla f_i(x_i^{t+1}) \right\|^2 \notag\\
=\; & \frac{p \sigma_i^2}{B_i} + (1 - p) \, \mathbb{E} \left\| y_i^t - \nabla f_i(x_i^t) \right\|^2 \notag\\
& + (1 - p) \, \mathbb{E} \left\| \frac{\omega_i^t}{b q} \sum_{\xi_{i,j} \in \mathcal{S}_i(t)} \left( g_i(x_i^{t+1}; \xi_{i,j}) - g_i(x_i^t; \xi_{i,j}) \right) - \left( \nabla f_i(x_i^{t+1}) - \nabla f_i(x_i^t) \right) \right\|^2 \notag\\
\leq\; & \frac{p \sigma_i^2}{B_i} + (1 - p) \, \mathbb{E} \left\| y_i^t - \nabla f_i(x_i^t) \right\|^2 
+ \frac{1 - p}{b q} \, \mathbb{E} \left\| g_i(x_i^{t+1}; \xi_{i,j}) - g_i(x_i^t; \xi_{i,j}) \right\|^2 \notag\\
\leq\; & \frac{p \sigma_i^2}{B_i} + (1 - p) \, \mathbb{E} \left\| y_i^t - \nabla f_i(x_i^t) \right\|^2 
+ \frac{3(1 - p)}{b q} \, \mathbb{E} \left\| g_i(x_i^{t+1}; \xi_{i,j}) - g_i(\bar{x}^{t+1}; \xi_{i,j}) \right\|^2 \notag \\
& + \frac{3(1 - p)}{b q} \, \mathbb{E} \left\| g_i(\bar{x}^{t+1}; \xi_{i,j}) - g_i(\bar{x}^{t}; \xi_{i,j}) \right\|^2
+ \frac{3(1 - p)}{b q} \, \mathbb{E} \left\| g_i(\bar{x}^{t}; \xi_{i,j}) - g_i(x_i^{t}; \xi_{i,j}) \right\|^2,
\end{align}
where the first equality follows from the update rule of the algorithm, the second step uses the variance bound in Assumption~\ref{ass_2.2}, the third step and the subsequent inequality apply variance decomposition and independence, and the last inequality follows from Young’s inequality.

Averaging the above result for $y_i(t+1)$ over $i = 1,\dots,m$ and using Assumption~\ref{ass_2.5} yields
\begin{align*}
& \mathbb{E} \| Y^{t+1} - \nabla \mathbf{f}(X^{t+1}) \|^2 \\
\leq\; & \sum_{i=1}^m \frac{p \sigma_i^2}{B_i} 
+ (1 - p) \, \mathbb{E} \| Y^t - \nabla \mathbf{f}(X^t) \|^2 \\
& + \frac{3(1 - p) m \bar{L}^2}{b q} \, \mathbb{E} \left[ \|X^{t+1} - \mathbf{1} \bar{x}^{t+1}\|^2 
+ \|\bar{x}^{t+1} - \bar{x}^t\|^2 
+ \|X^t - \mathbf{1} \bar{x}^t\|^2 \right] \\
\leq\; & \sum_{i=1}^m \frac{p \sigma_i^2}{B_i} 
+ (1 - p) \, \mathbb{E} \| Y^t - \nabla \mathbf{f}(X^t) \|^2 \\
& + \frac{3(1 - p) m \bar{L}^2}{b q} \, \mathbb{E} \left[ 
3 \|X^t - \mathbf{1} \bar{x}^t\|^2 
+ \|\bar{x}^{t+1} - \bar{x}^t\|^2 
+ 2 \eta^2 \|S^t - \mathbf{1} \bar{s}^t\|^2 \right],
\end{align*}
where the last inequality follows from \eqref{eq_b1_5} and the condition $\rho_t \le 1$.
Recalling the definition of $V(t)$, this completes the proof of the lemma.
\end{proof}

\begin{lemma}
\label{lemma_U}
Under the parameter choices specified in Theorem~\ref{thm_4.1}, for Algorithm~\ref{algo:d-nss-vr}, we have
\begin{equation*}
\mathbb{E}[U(t+1)] 
\leq (1 - p) \, \mathbb{E}[U(t)] 
+ \frac{9(1 - p)\bar{L}^2}{m b q} \, \mathbb{E}[C(t)] 
+ \frac{3(1 - p)\bar{L}^2}{m b q} \, \mathbb{E}\|\bar{x}^{t+1} - \bar{x}^t\|^2 
+ \sum_{i=1}^m \frac{p \sigma_i^2}{m^2 B_i}.
\end{equation*}
\end{lemma}

\begin{proof}
Noting that
\begin{equation*}
\mathbb{E}[U(t)] = \mathbb{E} \left\| \frac{1}{m} \sum_{i=1}^m y_i^t - \nabla f_i(x_i^t) \right\|^2 
= \frac{1}{m^2} \sum_{i=1}^m \mathbb{E} \left\| y_i^t - \nabla f_i(x_i^t) \right\|^2,
\end{equation*}
where the last equality holds since $\mathbb{E}[y_i^t] = \nabla f_i(x_i^t)$ for each $i$ and the variables $y_i^t$ are independent.
Combining this with the upper bound of $\mathbb{E}\| y_i^{t+1} - \nabla f_i(x_i^{t+1}) \|^2$ in \eqref{eq_b1_8} yields
\begin{equation*}
\mathbb{E}[U(t+1)] 
\leq (1 - p) \, \mathbb{E}[U(t)] 
+ \frac{9(1 - p)\bar{L}^2}{m b q} \, \mathbb{E}[C(t)] 
+ \frac{3(1 - p)\bar{L}^2}{m b q} \, \mathbb{E}\|\bar{x}^{t+1} - \bar{x}^t\|^2 
+ \sum_{i=1}^m \frac{p \sigma_i^2}{m^2 B_i},
\end{equation*}
which completes the proof of Lemma~\ref{lemma_U}.
\end{proof}

With the above lemmas established, we are now ready to prove Theorem~\ref{thm_4.1}.
\begin{proof}[Proof of Theorem~\ref{thm_4.1}]
Combining the results of Lemmas~\ref{lemma_descent}, \ref{lemma_C}, \ref{lemma_V}, and~\ref{lemma_U}, we obtain the following inequality:
\begin{align*}
& \mathbb{E}[\Phi(t+1)] \\
= \; & \mathbb{E}\left[f(\bar{x}^{t+1})+\frac{\eta}{p} U(t+1)+\frac{\eta}{m p} V(t+1)+\frac{1}{\eta} C(t+1)\right] \\
\leq \; & \mathbb{E}\left[f(\bar{x}^t)-\frac{\eta}{2}\|\nabla f(\bar{x}^t)\|^2-\left(\frac{1}{2 \eta}-\frac{\bar{L}}{2}\right)\|\bar{x}^{t+1}-\bar{x}^t\|^2+\eta U(t)+\bar{L}^2 \eta C(t)\right] \\
& +\frac{\eta}{p}\left[(1-p) \mathbb{E}[U(t)]+\frac{9(1-p) \bar{L}^2}{m b q} \mathbb{E}[C(t)]+\frac{3(1-p) \bar{L}^2}{m b q} \mathbb{E}\|\bar{x}^{t+1}-\bar{x}^t\|^2+\sum_{i=1}^m\frac{\eta \sigma_i^2}{m^2 B_i}\right] \\
& +\frac{\eta}{m p}\left[(1-p) \mathbb{E}[V(t)]+\frac{9(1-p) \bar{L}^2}{bq} \mathbb{E}[C(t)]+\frac{3(1-p) \bar{L}^2}{bq} \mathbb{E}\|\bar{x}^{t+1}-\bar{x}^t\|^2+\sum_{i=1}^m\frac{\eta \sigma_i^2}{m B_i}\right] \\
& +\frac{1}{\eta}\left[38 c m \rho_t^2 \mathbb{E}[C(t)]+6 \rho_t^2 \eta^2 m V(t) +24 c \rho_t^2 m^2\mathbb{E}\|\bar{x}^{t+1}-\bar{x}^t\|^2+\sum_{i=1}^m\frac{6 p \rho_t^2 \eta^2 \sigma_i^2}{B_i}\right] \\
= \; & \mathbb{E}\left[f(\bar{x}^t)-\frac{\eta}{2}\|\nabla f(\bar{x}^t)\|^2+\left(\eta+\frac{\eta(1-p)}{p}\right) U(t)+\left(\frac{\eta(1-p)}{m p}+6 \rho_t^2 \eta m \right) V(t)\right. \\
& +\left(\bar{L}^2 \eta+\frac{18(1-p) \eta \bar{L}^2}{m b q p}+\frac{38 c m \rho_t^2}{ \eta}\right) C(t)+\sum_{i=1}^m\left(\frac{2 \eta \sigma_i^2}{m^2 B_i}+\frac{6 p \rho_t^2 \eta \sigma_i^2}{B_i}\right) \\
& \left.-\left(\frac{1}{2 \eta}-\frac{\bar{L}}{2}-\frac{6(1-p) \bar{L}^2 \eta}{m b q p}-\frac{24 c \rho_t^2 m^2}{\eta}\right)\|\bar{x}^{t+1}-\bar{x}^t\|^2\right] \\
\leq \; & \mathbb{E}\left[\Phi(t)-\frac{\eta}{2}\|\nabla f(\bar{x}^t)\|^2-\frac{1}{2 \eta} C(t)+\sum_{i=1}^m\frac{3 \eta \sigma_i^2}{m^2 B_i}\right],
\end{align*}
where the last inequality follows from the parameter settings in Theorem~\ref{thm_4.1}. Specifically, letting
\begin{align*}
& \eta = \frac{1}{48\bar{L}}, \quad
B_i = \max\left\{ \left\lceil \frac{64 \sigma_i \sum_{j=1}^m \sigma_j}{m^2\epsilon^2} \right\rceil, 1 \right\}, \quad b = \left\lceil \frac{\sqrt{\sum_{i=1}^m B_i}}{m} \right\rceil,\quad
q = \frac{\sqrt{\sum_{i=1}^m B_i}}{b m},\quad p = \frac{b q}{b q + \sum_{i=1}^m B_i / m},\\
& T = \left\lceil \frac{384 \Delta \bar{L}}{\epsilon^2} + \frac{2}{p} \right\rceil,\quad R_0 = \left\lceil \frac{2+\sqrt{2}}{2\sqrt{\chi}}\log\left( 1 + \frac{16m \sum_{i=1}^m \|y_i^0\|^2}{T \epsilon^2} \right) \right\rceil, \quad R_t = \left\lceil \frac{2+\sqrt{2}}{2\sqrt{\chi}}\log\left( 1344cm^2\right)\right\rceil,
\end{align*}
we have 
\begin{align*}
& \frac{\eta(1-p)}{m p}+6 \rho_t^2 \eta m \le \frac{\eta}{mp} \\
& \bar{L}^2 \eta+\frac{18(1-p) \eta \bar{L}^2}{m b q p}+\frac{38 c m \rho_t^2}{ \eta} \le \frac{1}{2\eta} \\
& \frac{6 p \rho_t^2 \eta \sigma_i^2}{B_i} \le \frac{\eta \sigma_i^2}{m^2 B_i} \\
& \frac{1}{2 \eta}-\frac{\bar{L}}{2}-\frac{6(1-p) \bar{L}^2 \eta}{m b q p}-\frac{24 c \rho_t^2 m^2}{\eta} \ge 0.
\end{align*}

Summing over $t = 0,\dots,T-1$ yields
\begin{align}
\label{eq_b2_1}
\mathbb{E}\left[\frac{1}{T} \sum_{t=0}^{T-1}\|\nabla f(\bar{x}^t)\|^2\right] 
\leq\; & \frac{2}{\eta T} \sum_{t=0}^{T-1} \mathbb{E}\left[\Phi(t)-\Phi(t+1)-\frac{1}{2 \eta} C(t)+\sum_{i=1}^m \frac{3 \eta \sigma_i^2}{m^2 B_i}\right] \notag\\
\leq\; & \frac{2}{\eta T} \sum_{t=0}^{T-1} \mathbb{E}\left[\Phi(t)-\Phi(t+1)-\frac{1}{2 \eta}\|X^t - \mathbf{1} \bar{x}^t\|^2+\sum_{i=1}^m \frac{3 \eta \sigma_i^2}{m^2 B_i}\right] \notag\\
=\; & \frac{2 \mathbb{E}[\Phi(0)-\Phi(T)]}{\eta T} - \frac{1}{\eta^2 T} \sum_{t=0}^{T-1} \|X^t - \mathbf{1} \bar{x}^t\|^2 + \sum_{i=1}^m \frac{6 \sigma_i^2}{m^2 B_i},
\end{align}
where the second inequality follows from the definition of $C(t)$.

For $\mathbb{E}[\Phi(0)-\Phi(T)]$, we have
\begin{align}
\label{eq_b2_2}
& \mathbb{E}[\Phi(0)-\Phi(T)] \notag\\
=\; & \mathbb{E}\left[f(\bar{x}^0)-f(\bar{x}^T)+\frac{\eta}{p} U^0+\frac{\eta}{m p} V^0+\frac{1}{m \eta} C^0 - \frac{\eta}{p} U(T) - \frac{\eta}{m p} V(T) - \frac{1}{m \eta} C(T)\right] \notag\\
\leq\; & f(\bar{x}^0)-f^* + \sum_{i=1}^m \frac{2 \eta \sigma_i^2}{m^2 B_i p} + \frac{\eta}{m} \mathbb{E}\|S^0 - \mathbf{1} \bar{s}^0\|^2,
\end{align}
where the inequality follows from the nonnegativity of $U(T)$, $V(T)$, and $C(T)$, together with Assumption~\ref{ass_2.2}.

From the parameter choice we have $Tp \ge 2$. Substituting \eqref{eq_b2_2} into \eqref{eq_b2_1} gives
\begin{align*}
\mathbb{E}\left[\frac{1}{T} \sum_{t=0}^{T-1} \|\nabla f(\bar{x}^t)\|^2\right] 
\leq\; \frac{2 \Delta}{\eta T} + \frac{2}{m T} \mathbb{E} \|S^0 - \mathbf{1} \bar{s}^0\|^2 + \sum_{i=1}^m \frac{8 \sigma_i^2}{m^2 B_i} - \frac{1}{\eta^2 T} \sum_{t=0}^{T-1} \mathbb{E} \|X^t - \mathbf{1} \bar{x}^t\|^2.
\end{align*}

Therefore, the final output satisfies the following error bound:
\begin{align*}
& \quad \mathbb{E} \left[ \left\|\nabla f\left(x_{i,\text{out}}\right)\right\|^2 \right]\\
&= \frac{1}{T} \sum_{t=0}^{T-1} \mathbb{E}\left\|\nabla f\left(x_i^t\right)\right\|^2 \\
&\leq \frac{2}{T} \sum_{t=0}^{T-1} \mathbb{E}\left[\| \nabla f(\bar{x}^t) \|^2 + \| \nabla f(x_i^t) - \nabla f(\bar{x}^t) \|^2\right] \\
&\leq \frac{2}{T} \sum_{t=0}^{T-1} \mathbb{E}\left[\|\nabla f(\bar{x}^t)\|^2 + \bar{L}^2 \left\|x_i^t - \bar{x}^t\right\|^2\right] \\
&\leq \frac{2}{T} \sum_{t=0}^{T-1} \mathbb{E} \|\nabla f(\bar{x}^t)\|^2 + \frac{2 \bar{L}^2}{T} \sum_{t=0}^{T-1} \mathbb{E} \|X^t - \mathbf{1} \bar{x}^t\|^2 \\
&\leq \frac{4 \Delta}{\eta T} + \frac{4}{m T} \mathbb{E}\|S^0 - \mathbf{1} \bar{s}^0\|^2 + \sum_{i=1}^m \frac{16 \sigma_i^2}{m^2 B_i} - \left( \frac{2}{\eta^2 T} - \frac{2 \bar{L}^2}{T} \right) \sum_{t=0}^{T-1} \mathbb{E} \|X^t - \mathbf{1} \bar{x}^t\|^2 \\
&\leq \frac{4 \Delta}{\eta T} + \frac{4}{m T} \mathbb{E}\|S^0 - \mathbf{1} \bar{s}^0\|^2 + \sum_{i=1}^m \frac{16 \sigma_i^2}{m^2 B_i} \\
& \leq \epsilon^2,
\end{align*}
where the last inequality holds by the selection of parameters.

The sampling complexity of the algorithm is bounded by
\begin{align*}
& \quad \sum_{i=1}^m B_i + \left(p\sum_{i=1}^m B_i + (1-p)qbm\right)T \\
& = O \left( m + \frac{\bar{\sigma}_{\text{AM}}^2}{\epsilon^2} + \left( \sqrt{m} + \frac{\bar{\sigma}_{\text{AM}}}{\epsilon} \right) \left( \frac{\Delta \bar{L}}{\epsilon^2} + \sqrt{m} + \frac{\bar{\sigma}_{\text{AM}}}{\epsilon}\right) \right)\\
& = O \left( \frac{\Delta \bar{L} \bar{\sigma}_{\text{AM}}}{\epsilon^3} + \frac{\bar{\sigma}_{\text{AM}}^2}{\epsilon^2}  + \sqrt{m}\frac{\Delta \bar{L}}{\epsilon^2} + m \right),
\end{align*}
and the communication complexity is bounded by
\begin{align*}
    \sum_{i=0}^{T-1} R_i = \tilde{O}\left( \dfrac{\Delta \bar{L}}{\sqrt{\chi} \epsilon^2} + \dfrac{\bar{\sigma}_{\text{AM}}}{\sqrt{\chi} \epsilon}\right).
\end{align*}
This completes the proof of Theorem~\ref{thm_4.1}.
\end{proof}

\section{The Proof of Theorem~\ref{thm_4.2}}
For the lower bound on the complexity of variance reduction methods, we continue to use the function $F_D$ defined in \eqref{eq_FD}. However, the stochastic gradient constructed in \eqref{eq_G1} does not satisfy the mean-square smoothness Assumption~\ref{ass_2.5}. Therefore, a new stochastic gradient needs to be constructed. Consider the following stochastic gradient \cite{arjevani2023lower}:
\begin{align*}
    [\tilde{G}_D(x, \xi; p)]_i &:= \nabla_i F_D(x) \left( 1 + \Theta_i(x) \left( \frac{a}{p} - 1 \right) \right),
\end{align*}
where $\xi \sim \text{Bernoulli}(p)$,
\begin{align*}
    \Theta_i(x) &= \Gamma \left( 1 - \left( \frac{1}{T} \sum_{k=i}^T \Gamma^2 (|x_k|) \right)^{\frac{1}{2}} \right), \quad \Gamma(t) = \frac{\int_0^t \Lambda(\tau) d\tau}{\int_0^{1/2} \Lambda(\tau') d\tau'},
\end{align*}
and 
\begin{align*}
    \Lambda(t) = \begin{cases}
        0, & t \geq \frac{1}{2}; \\
        \exp\left( -\frac{1}{100 + t (1/2 - t)} \right), & 0 < t < \frac{1}{2}.
    \end{cases}
\end{align*}
The stochastic gradient $\tilde{G}_D$ satisfies the properties stated in the following lemma.
\begin{lemma}[{\citet{arjevani2023lower}}]
\label{lemma_2.8}
The stochastic gradient $\tilde{G}_D(x,\xi;p)$ satisfies:
\begin{enumerate}
    \item $\mathbb{E}_\xi [\tilde{G}_D(x, \xi; p)] = \nabla F_D(x)$.
    \item For any $x \in \mathbb{R}^D$, it holds that 
    \begin{align*}
        \mathbb{E}_\xi \|\tilde{G}_D(x, \xi) - \nabla F_D(x)\|^2 \leq \frac{a^2 (1 - p)}{p},
    \end{align*}
    where $a = 23$. 
    \item For any $x, y \in \mathbb{R}^D$,
    \begin{align*}
    \mathbb{E}_\xi \|\tilde{G}_D(x, \xi; p) - \tilde{G}_D(y, \xi; p)\|^2 \leq \frac{\ell_2^2}{p} \|x - y\|^2,
    \end{align*}
    where $\ell_2 = 328$. 
\end{enumerate}
\end{lemma}

For variance reduction methods, we can establish the sample complexity lower bound in Theorem~\ref{thm_4.2}.
\begin{proof}[Proof of Theorem~\ref{thm_4.2}]
Let $\#\text{SFO}$ denote the total number of stochastic first-order oracle calls of the algorithm, and $\#\text{SFO}_i$ denote the number of oracle calls performed at node $i$. For any algorithm $\mathcal{A}$, we have
\begin{align*}
\mathbb{E} \left[ \#\text{SFO}_i \right] = c_i \cdot \#\text{SFO}, \quad \text{where} \quad c_i > 0 \quad \text{and} \quad \sum_{i=1}^m c_i = 1.
\end{align*}
Letting $D = \sum_i D_i$, we define a matrix sequence $\{U_i\}_{i=1}^m$ such that $U_i \in \mathbb{R}^{D_i\times D}$, $U_i U_i^\top = I$, and $U_i U_j^\top = \mathbf{0}$ for any $1\le i \ne j \le m$.

The functions are constructed as
\begin{align*}
f_i(x) &= \frac{\sqrt{m p_i} \, \bar{L} \lambda_i^2}{\ell} \, F_{D_i} \left( \frac{U_i x}{\lambda_i} \right), \;
g_i(x, \xi; p_i) = \frac{\sqrt{m p_i} \, \bar{L} \lambda_i}{\ell} \, U_i^\top \tilde{G}_{D_i} \left( \frac{U_i x}{\lambda_i}, \xi; p_i \right), \\
\text{where} & \quad \ell = \max\{\ell_1, \ell_2\}, \quad \lambda_i  = \frac{\ell}{\sqrt{p_i} \bar{L}} 
\left( \frac{\sigma_i}{\bar{\sigma}_{2/3}} \right)^{\tfrac{1}{3}} 
\cdot 2 \epsilon, \quad
\frac{1}{p_i} = \frac{\left( \sigma_i^2 \bar{\sigma}_{2/3} \right)^{2/3}}{4 m a^2 \epsilon^2} + 1. 
\end{align*}
First, we verify that the constructed functions satisfy the mean-square smoothness Assumption~\ref{ass_2.5}. It holds that
\begin{align*}
& \quad \frac{1}{m} \sum_{i=1}^m \mathbb{E} \left[ \left\| g_i(x, \xi; p_i) - g_i(y, \xi; p_i) \right\|^2 \right] \\
&= \frac{1}{m} \sum_{i=1}^m \mathbb{E} \left[ \left\| \frac{\sqrt{m p_i} \bar{L} \lambda_i}{\ell} 
U_i^\top \left( \tilde{G}_{D_i} \left( \frac{U_i x}{\lambda_i}, \xi; p_i \right) 
- \tilde{G}_{D_i} \left( \frac{U_i y}{\lambda_i}, \xi; p_i \right) \right) \right\|^2 \right] \\
&= \sum_{i=1}^m \frac{p_i \bar{L}^2 \lambda_i^2}{\ell^2} 
\mathbb{E} \left[ \left\| \tilde{G}_{D_i} \left( \frac{U_i x}{\lambda_i}, \xi; p_i \right) 
- \tilde{G}_{D_i} \left( \frac{U_i y}{\lambda_i}, \xi; p_i \right) \right\|^2 \right] \\
&\le \sum_{i=1}^m \frac{p_i \bar{L}^2 \lambda_i^2}{\ell^2} \cdot 
\frac{\ell^2}{p_i} \left\| \frac{U_i x}{\lambda_i} - \frac{U_i y}{\lambda_i} \right\|^2 \\
&= \bar{L}^2 \sum_{i=1}^m \left\| U_i (x - y) \right\|^2 \\
&= \bar{L}^2 \left\| x - y \right\|^2.
\end{align*}
Then we verify that the stochastic gradients satisfy Assumption~\ref{ass_2.2}. By Property~1 of Lemma~\ref{lemma_2.8}, $g_i(x,\xi;p_i)$ is unbiased, and the variance has the following upper bound:
\begin{align*}
& \quad \mathbb{E} \left[ \left\| g_i(x, \xi; p_i) - \nabla f_i(x) \right\|^2 \right] \\
&= \mathbb{E} \left[ \left\| \frac{\sqrt{m p_i} \bar{L} \lambda_i}{\ell} 
U_i^\top \left( \tilde{G}_{D_i} \left( \frac{U_i x}{\lambda_i}, \xi; p_i \right) 
- \nabla F_{D_i} \left( \frac{U_i x}{\lambda_i} \right) \right) \right\|^2 \right] \\
&= \frac{m p_i \bar{L}^2 \lambda_i^2}{\ell^2} 
\mathbb{E} \left[ \left\| \tilde{G}_{D_i} \left( \frac{U_i x}{\lambda_i}, \xi; p_i \right) 
- \nabla F_{D_i} \left( \frac{U_i x}{\lambda_i} \right) \right\|^2 \right] \\
&\le 4 m \epsilon^2 
\left( \frac{\sigma_i}{\bar{\sigma}_{2/3}} \right)^{2/3} \cdot 
\frac{a^2 (1 - p_i)}{p_i} \\
&= 4 m \epsilon^2
\left( \frac{\sigma_i}{\bar{\sigma}_{2/3}} \right)^{2/3} \cdot 
\frac{a^2 \left( \sigma_i^2 \bar{\sigma}_{2/3} \right)^{2/3}}{4 m a^2 \epsilon^2} \\
&= \sigma_i^2,
\end{align*}
where the inequality holds by the Property~2 of Lemma~\ref{lemma_2.8}.

Letting $D_i = \left\lfloor \dfrac{\sqrt{m} \Delta \sqrt{p_i} \ell}{4 \Delta_0 \ell \epsilon^2} 
\left( \dfrac{\bar{\sigma}_{2/3}}{\sigma_i} \right)^{2/3} c_i \right\rfloor$, we have
\begin{align*}
f(0) - \inf_x f(x)
&= \frac{1}{m} \sum_{i=1}^m f_i(0) - \inf_x \frac{1}{m} \sum_{i=1}^m f_i(x) \\
&\le \frac{1}{m} \sum_{i=1}^m f_i(0) - \frac{1}{m} \sum_{i=1}^m \inf_x f_i(x) \\
&= \frac{1}{\sqrt{m}} \sum_{i=1}^m \frac{\sqrt{p_i} \bar{L} \lambda_i^2}{\ell} 
\left( F_{D_i}(0) - \inf_x F_{D_i} \left( \frac{U_i x}{\lambda_i} \right) \right) \\
&= \frac{4 \epsilon^2 \ell}{ \sqrt{m} \bar{L}} 
\sum_{i=1}^m \left( \frac{\sigma_i}{\bar{\sigma}_{2/3}} \right)^{2/3} \cdot 
\frac{1}{\sqrt{p_i}} 
\left( F_{D_i}(0) - \inf_x F_{D_i} \left( \frac{U_i x}{\lambda_i} \right) \right) \\
&\le \frac{4 \epsilon^2 \ell}{ \sqrt{m} \bar{L}} 
\sum_{i=1}^m \left( \frac{\sigma_i}{\bar{\sigma}_{2/3}} \right)^{2/3} \cdot 
\frac{1}{\sqrt{p_i}} \cdot \Delta_0 \cdot 
\frac{ \sqrt{m} \Delta \sqrt{p_i} \bar{L}}{4 \Delta_0 \ell \epsilon^2} 
\left( \frac{\bar{\sigma}_{2/3}}{\sigma_i} \right)^{2/3} c_i \\
&= \Delta \sum_{i=1}^m c_i \\
&= \Delta.
\end{align*}
Therefore, under this choice of $D_i$, Assumption~\ref{ass_2.1} is satisfied. 
This completes the verification that the constructed functions satisfy Assumptions 1, 2, and 5.

If $\#\mathrm{SFO} \le \dfrac{1}{64 \Delta_0 \ell a} \cdot \dfrac{\Delta \bar{L} \bar{\sigma}_{2/3}}{\epsilon^3}$, and $\epsilon \le \min_i \sqrt{\dfrac{\sqrt{m} \Delta \sqrt{p_i} \ell}{8 \Delta_0 \ell \epsilon^2} 
\left( \dfrac{\bar{\sigma}_{2/3}}{\sigma_i} \right)^{2/3} c_i}$, then we have
\begin{align*}
& \quad \mathbb{E} \left[ \left\| \nabla f(x) \right\|^2 \right] \\
&= \mathbb{E} \left[ \left\| \frac{1}{m} \sum_{i=1}^m \nabla f_i(x) \right\|^2 \right] \\
&= \mathbb{E} \left[ \left\| \frac{1}{\sqrt{m}} \sum_{i=1}^m \frac{\sqrt{p_i} \bar{L} \lambda_i^2}{\ell} U_i^\top \nabla F_{D_i} \left( \frac{U_i x}{\lambda_i} \right) \right\|^2 \right] \\
&= \frac{4 \epsilon^2}{m} \sum_{i=1}^m \left( \frac{\sigma_i}{\bar{\sigma}_{2/3}} \right)^{2/3} 
\mathbb{E} \left[ \left\| \nabla F_{D_i} \left( \frac{U_i x}{\lambda_i} \right) \right\|^2 \right] \\
&> \frac{4 \epsilon^2}{m} \sum_{i=1}^m \left( \frac{\sigma_i}{\bar{\sigma}_{2/3}} \right)^{2/3} 
\mathbb{P} \left( [U_i x]_{D_i} = 0 \right) \\
&= \frac{4 \epsilon^2}{m} \sum_{i=1}^m \left( \frac{\sigma_i}{\bar{\sigma}_{2/3}} \right)^{2/3}
\mathbb{P} \left( [U_i x]_{D_i} = 0 \,\middle|\, \#\mathrm{SFO}_i < \frac{D_i - 1}{2 p_i} \right)
\mathbb{P} \left( \#\mathrm{SFO}_i < \frac{D_i - 1}{2 p_i} \right) \\
&\ge \frac{4 \epsilon^2}{m} \sum_{i=1}^m \left( \frac{\sigma_i}{\bar{\sigma}_{2/3}} \right)^{2/3} \cdot \frac{1}{2} \cdot \frac{1}{2} \\
&= \epsilon^2.
\end{align*}
The final inequality above follows from Lemma~\ref{lemma_2.3} and the following probability lower bound:
\begin{align*}
\mathbb{P} \left( \#\mathrm{SFO}_i < \frac{D_i - 1}{2p_i} \right)
&= 1 - \mathbb{P} \left( \#\mathrm{SFO}_i \ge \frac{D_i - 1}{2p_i} \right) \\
&\ge 1 - \frac{ \mathbb{E}[\#\mathrm{SFO}_i] }{ \dfrac{D_i - 1}{2p_i} } \\
&= 1 - \frac{ \dfrac{1}{64 \Delta_0 \ell a} \cdot \dfrac{\Delta \bar{L} \bar{\sigma}_{2/3}}{\epsilon^3} \cdot c_i }
{ \dfrac{ \sqrt{m} \Delta \bar{L} }{ 8 \Delta_0 \ell \epsilon^2 } 
\left( \dfrac{ \bar{\sigma}_{2/3} }{ \sigma_i } \right)^{2/3}
c_i \cdot \dfrac{ (\sigma_i^2 \bar{\sigma}_{2/3})^{1/3}  }
{ 4 \sqrt{m} a \epsilon } } \\
&= \frac{1}{2},
\end{align*}
where the first inequality follows from Markov’s inequality.
Therefore, to ensure $\mathbb{E}\left[\|\nabla f(x)\|^{2}\right] \le \epsilon^{2}$, the number of samples must be at least $\Omega(\Delta \bar{L} \bar{\sigma}_{2/3} \epsilon^{-3})$.

For the second term in Theorem~\ref{thm_4.2}, we construct the following functions:
\begin{align*}
    f_i(x, \xi_i) = \frac{\sqrt{m} \bar{L}}{2} \left( \left\| U_i x \right\|^2 
- 2 \xi_i \left[ U_i x \right]_1 + r^2 \right),
\end{align*}
where $r \in \left( 0, \sqrt{ \frac{2\Delta }{\sqrt{m} \bar{L}} } \right), \quad 
\xi_i \sim \mathcal{P}^{s}_{\xi_i} := \mathcal{N} \left( r s, \frac{ \sigma_i^2}{m \bar{L}^2} \right), \quad
s \in \{-1, 1\}.$

Let $\theta_s = [r s, 0, \ldots, 0]^\top \in \mathbb{R}^d$, and $f_i^s(x) = \mathbb{E}_{\xi_i} \left[ f_i(x, \xi_i) \right] = \frac{\sqrt{m} \bar{L}}{2} \left\| U_i x - \theta_s \right\|^2$. 
It is straightforward to verify that
\begin{align*}
f(0) - \inf_x f(x) 
&= \frac{1}{m} \sum_{i=1}^m f_i(0) - \inf_x \frac{1}{m} \sum_{i=1}^m f_i(x) \\
&\le \frac{1}{m} \sum_{i=1}^m \left( f_i(0) - \inf_x f_i(x) \right) \\
&= \frac{1}{m} \sum_{i=1}^m \frac{\sqrt{m} \bar{L}}{2} r^2 \\
&\le \Delta.
\end{align*}
Since
\begin{align*}
\mathbb{E}_{\xi_i} \left[ \nabla_x f_i(x, \xi_i) - \nabla f_i(x) \right]
&= \sqrt{m} \bar{L} \, \mathbb{E}_{\xi_i} \left[ (\xi_i - rs) \right] \, U_i^\top e_i = 0, 
\end{align*}
and
\begin{align*}
\mathbb{E}_{\xi_i} \left[ \left\| \nabla_x f_i(x, \xi_i) - \nabla f_i(x) \right\|^2 \right]
&= m \bar{L}^2 \, \mathbb{E}_{\xi_i} \left[ (\xi_i - rs)^2 \right] \left\| U_i^\top e_i \right\|^2 = \sigma_i^2,
\end{align*}
the constructed function satisfies Assumption~\ref{ass_2.2}. 
Moreover, 
\begin{align*}
\frac{1}{m} \sum_{i=1}^m \mathbb{E}_{\xi_i} \left[ \left\| \nabla f_i(x, \xi_i) - \nabla f_i(y, \xi_i) \right\|^2 \right]
&= \frac{1}{m} \sum_{i=1}^m \mathbb{E}_{\xi_i} \left[ \left\| \sqrt{m} \bar{L} U_i^\top (U_i x - U_i y) \right\|^2 \right] \\
&= \sum_{i=1}^m \bar{L}^2 \left\| U_i^\top (U_i x - U_i y) \right\|^2 \\
&= \sum_{i=1}^m \bar{L}^2 \left\| U_i x - U_i y \right\|^2 \\
&= \bar{L}^2 \left\| x - y \right\|^2.
\end{align*}
so it also satisfies Assumption~\ref{ass_2.5}.

For node $i$, define $\hat{s} = \arg\min_{s' \in \{1, -1\}} \left\| \nabla f_i^{s'}(x) \right\|$. When $\hat{s} \ne s$, we have
\begin{align*}
\left\| \nabla f_i^{s}(x) \right\| 
&\ge \frac{1}{2} \inf_x \left\{ \left\| \nabla f_i^{1}(x) \right\| + \left\| \nabla f_i^{-1}(x) \right\| \right\} \\
&= \sqrt{m} \bar{L} \inf_x \left\{ \left\| U_i x -  \theta_1 \right\| + \left\| U_i x -  \theta_{-1} \right\| \right\} \\
&\ge \sqrt{m} \bar{L} \left\|  \theta_1 - \theta_{-1} \right\| \\
&= \sqrt{m} \bar{L} r.
\end{align*}
By Markov's inequlity, we further have
\begin{align}
\label{eq_vrlb_1}
\mathbb{E} \left[ \left\| \nabla f_i^s(x) \right\| \right] 
&\ge \sqrt{m} \bar{L} r \cdot \mathbb{P} \left( \left\| \nabla f_i^s(x) \right\| \ge \sqrt{m} L r \right) \notag\\
&\ge \sqrt{m} \bar{L} r \cdot \mathbb{P} \left( \hat{s} \ne s \right)
\end{align}

For node $i$, letting $\mathbb{P}_s^D = \mathcal{N}^{\otimes D} \left( r s, \frac{ \sigma_i^2}{ m \bar{L}^2} \right)$, it holds that
\begin{align*}
\mathbb{P}(\hat{s} \ne s) &\ge \frac{1}{2} - \frac{1}{2} \left\| \mathbb{P}_1^D - \mathbb{P}_{-1}^D \right\|_{\mathrm{TV}} \\
&\ge \frac{1}{2} - \frac{1}{2} \sqrt{ \frac{1}{2} D_{\mathrm{KL}}(\mathbb{P}_1^D \, \| \, \mathbb{P}_{-1}^D) } \\
&= \frac{1}{2} \left( 1 - \frac{\sqrt{m} \bar{L} r \sqrt{D}}{\sigma_i} \right)
\end{align*}

Combining with \eqref{eq_vrlb_1}, this yields
\begin{align*}
\mathbb{E} \left[ \left\| \nabla f_i^s(x) \right\| \right]
\ge \frac{\sqrt{m}}{2} \, \bar{L} r \left( 1 - \frac{ \sqrt{m} \bar{L} r \sqrt{D} }{ \sigma_i } \right).
\end{align*}

Letting $r = \min \left\{ \frac{ \sigma_i }{ 2 \sqrt{m} \, \bar{L} \sqrt{D} }, \; \sqrt{ \frac{ 2\Delta }{ \sqrt{m} \bar{L} } } \right\}$ and
$\epsilon \le \sqrt{\frac{\sqrt{m}\Delta \bar{L}}{8}} \frac{\bar{\sigma}_{\text{AM}}}{\sigma_i}\sqrt{c_i}$, we have
\begin{align*}
\mathbb{E} \left[ \left\| \nabla f_i^s(x) \right\| \right] 
&\ge \min \left\{ \frac{\sigma_i}{8 \sqrt{D}}, \; \sqrt{\frac{ \sqrt{m} \Delta \bar{L} }{8}} \right\} \\
&\ge \min \left\{ \frac{\sigma_i}{8 \sqrt{D}}, \; \frac{\sigma_i}{\bar{\sigma}_{\text{AM}}\sqrt{ c_i }} \, \epsilon \right\}.
\end{align*}

Therefore, for node $i$, at least $\dfrac{\bar{\sigma}_{\text{AM}^2}}{64 m \epsilon^2} c_i$ samples are required to ensure $\mathbb{E} \left[ \left\| \nabla f_i^s(x) \right\| \right] \le \dfrac{\sigma_i}{\bar{\sigma}_{\text{AM}}\sqrt{ c_i }} \epsilon$.

If the number of total samples $\#\text{Sample}\le \dfrac{ \bar{\sigma}_{\text{AM}}^2}{128\epsilon^2}$, then
\begin{align*}
\mathbb{E} \left[ \left\| \nabla f(x) \right\|^2 \right] 
&= \mathbb{E} \left[ \left\| \frac{1}{m} \sum_{i=1}^m \nabla f_i(x) \right\|^2 \right] \\
&= \frac{1}{m^2} \sum_{i=1}^m \mathbb{E} \left[ \left\| \nabla f_i(x) \right\|^2 \right] \\
&\ge \frac{1}{m^2} \sum_{i=1}^m \left( \mathbb{E} \left[ \left\| \nabla f_i(x) \right\| \right] \right)^2 \\
&\ge \frac{1}{m^2} \sum_{i=1}^m \frac{\sigma_i^2}{\bar{\sigma}_{\text{AM}}^2 c_i } \epsilon^2 \\
&\ge \epsilon^2,
\end{align*}
where the last inequality holds by Cauchy–Schwarz inequality.

Therefore, we obtain the second term $\Omega\left( \bar{\sigma}_{\text{AM}}^2\epsilon^{-2}\right)$ in the lower bound on sample complexity in Theorem~\ref{thm_4.2}.

For the third and fourth terms in Theorem~\ref{thm_4.2}, which are independent of the variance, the construction of stochastic gradients is not required. These terms reduce to the lower bounds for variance reduction methods in the deterministic setting. According to the deterministic lower bounds for nonconvex finite-sum problems established by \citet{zhou2019lower}, the sample complexity is lower bounded by $\Omega\left(\sqrt{m}\Delta\bar{L}\epsilon^{-2} + m\right)$.
\end{proof}

\end{document}